\documentclass[12pt]{amsart}

\usepackage{amsthm,graphicx, amssymb, enumerate,verbatim}
\usepackage{subfig}
\usepackage{amsmath, amscd,mathabx}
\usepackage{bbm}
\usepackage{mathrsfs}
\usepackage{latexsym}
\usepackage{color}
\usepackage[margin=1.25 in]{geometry}

\pdfoutput=1

\usepackage{enumitem}
\setlist[itemize]{leftmargin=*}
\setlist[enumerate]{leftmargin=*}

\usepackage{epsfig}

\theoremstyle{definition}
\numberwithin{equation}{section}

\newtheorem{theorem}{Theorem}[section] 
\newtheorem{lemma}[theorem]{Lemma}

\theoremstyle{definition}
\newtheorem{prop}[theorem]{Proposition}
\newtheorem{defn}[theorem]{Definition}

\newtheorem{remark}[theorem]{Remark}

\newcommand{\R}{\mathbb R}
\newcommand{\C}{\mathbb C}

\def\P{{\mathbb P}}
\def\rp{{\R\P^2}}

\newcommand{\ds}{\displaystyle}

\def\phi{\varphi}

\begin{document}

\title[Herman rings]{An experimental view of Herman rings for dianalytic maps of $\rp$}

\author{Jane Hawkins }
\address{ Department of Mathematics \\ University of North Carolina
at Chapel Hill\\ CB \#3250 \\ Chapel Hill, North Carolina
27599-3250} 
 \email{jmh@math.unc.edu}

\author{Michelle Randolph}
\address{Department of Mathematics\\ University of North Carolina
at Chapel Hill\\ CB \#3250 \\ Chapel Hill, North Carolina
27599-3250} 
 \email{randolph.michelle@gmail.com}



\date{\today}

\maketitle

\begin{abstract}
We provide an experimental study of the existence of Herman rings in a parametrized family of rational maps preserving antipodal points, and a discussion of their properties on 
$\mathbb{RP}^2$. We study analytic maps of the sphere that project to dianalytic maps on the nonorientable surface, $\mathbb{RP}^2$. They have a known form and we focus on a subset of degree $3$ dianalytic maps and explore their dynamical properties. In particular, we focus on maps for which the Fatou set has a Herman ring. We appeal to dynamical properties of particular maps to justify our assertion that these Fatou components are Herman rings and analyze the parameter space for this family of maps. 

\end{abstract}

\listoffigures

\section{Introduction}\label{introduction}
For a rational map $f$ of the Riemann sphere, a Herman ring is a Fatou component on which $f$ (or an iterate of $f$) is analytically conjugate to a Euclidean rotation of some annulus onto itself. Herman rings, first constructed by Michael Herman,  are impossible for polynomials \cite{Milnor1}. There is a methods of constructing them due to Shishikura involving quasiconformal surgery in which he takes rational functions with Siegel disks, cuts them along their invariant curves, and glues them back together in a particular way that results in a rational map with Herman rings, see \cite{Surgery}. There is also a less well known method by Wang and Zang given in \cite{Twisting} in which they put topological conditions on a map $F$ and create a map $G$ with a Herman ring such that $F$ and $G$ are topologically equivalent.

This paper approaches finding Herman rings in a different way by exploring dianalytic maps of degree 3. Dianalytic maps are maps on the real projective plane, $\mathbb{RP}^2$, which lift to analytic maps of the double cover $\mathbb{C}_{\infty}$. These maps were first mentioned for bicritical maps by Milnor in \cite{Milnor2}, and a general form for dianalytic maps is given by Barza and Ghisa in \cite{BarzaGhisa}. The motivation for this project comes largely from a paper by Goodman and Hawkins \cite{GHawk}, studying bicritical rational maps that project to unicritical maps on $\mathbb{RP}^2$. While many of their results apply to maps of any odd degree, their results focus on the degree $3$ case including proving results about dynamical properties of their maps for some subsets of their parameter space. A perturbation of these maps gives new maps that are no longer bicritical and  we study the dynamical properties of them. An example of this form was mentioned in the last section of their paper but not fully discussed; it is developed in a nonrigourous way in this study, and was the topic of the master's project of the second author \cite{Rand}.  For bicritical maps, the Fatou set cannot contain a Herman ring since all Julia sets are connected \cite{Milnor2,GHawk}, however when we perturb away from the bicritical case, we find maps that we claim  have Herman rings.

While Shishikura's results imply the existence of rational maps with both a Herman ring and periodic orbits of period $>1$, we believe the maps described here are some of the first explicit examples of such maps. Bonifant, Buff, and Milnor have also studied antipodal-preserving cubic maps, specifically antipode preserving cubic maps with critical fixed points at $0$ and $\infty$ \cite{preprint}. They also find Herman rings in this class of maps, discussed in a preprint from 2015 \cite{preprint}, available at the Stony Brook IMS website. Their constructions have some fundamental differences with ours.   Results about Herman rings are typically very difficult to prove; therefore, the claims in this paper are largely based on experimental data and are presented without rigorous proof. We attempt to be as explicit as possible about why we believe these claims to be correct. We state many known results without proof, and give proofs of statements where needed or possible.  

In Section \ref{sec:dianalytic}, we give the general form for dianalytic maps as well as a few results about them. 
We also discuss some of the results from Goodman and Hawkins that serve as the inspiration for this project.
In Section \ref{sec:examples}, we introduce the form of the maps that are the focus of this work and show examples of these maps with Herman rings. We discuss the orientability of the rings when projected onto $\mathbb{RP}^2$. We also attempt to justify as fully as possible why we believe these maps possess Herman rings using many of the results from the previous sections. 

In Section \ref{sec:parameters}, we focus on the parameter space for our maps and show that it is symmetric across both the real and imaginary axis. 
\section{Background on Complex dynamics}\label{sec:background}  A reader familiar with the basics of complex dynamics can skip this section and move to Section \ref{sec:dianalytic} as this section includes definitions and results about iteration of rational maps that are relevant for this paper. Unless otherwise noted, all theorems and definitions are based on Beardon's book \emph{Iteration of Rational Functions} \cite{Beardon}. 
Throughout, we denote the composition of $f$ with $g$ (i.e. the mapping $x\mapsto f(g(x))$) interchangeably by $f\circ g(x)$, $f \circ g$, $fg(x)$, or simply $fg$. 
and  $\mathbb{C}_{\infty}$ denotes the Riemann Sphere. Let $R:\mathbb{C}_{\infty}\rightarrow \mathbb{C}_{\infty}$ be a rational map, that is $R(z)=\frac{a_0+a_1z+...+a_nz^n}{b_0+b_1z+...+b_mz^m}$ where $\deg R = \max\{m,n\} \geq 2$ and denote the $n$-fold composition of $R$ with itself by $R^n=\underbrace{R\circ R \circ ... \circ R}_{\text{n times}}.$ 

\begin{defn} 
	A point $\zeta\in \mathbb{C}_{\infty}$ is a \emph{periodic 	point} if $\zeta$ is fixed by some iterate of $R$, so $R^n(\zeta)=\zeta$ for some $n\in\mathbb{N}$. The least such $n$ is called the \emph{period} of $\zeta$ and we say that $\{\zeta, R(\zeta), R^2(\zeta),...,R^{n-1}(\zeta)\}$ is a \emph{cycle} of period $n$. In particular, $\zeta$ is a \emph{fixed point} of $R$ if it satisfies $R(\zeta)=\zeta$, so $\zeta$ is a periodic point of period $1$.  
\end{defn}
\begin{defn}\label{fixed points}
	Let $\zeta \in \mathbb{C}$ be a fixed point of $R$. Then the derivative $R'(\zeta)$ is defined, and we say that $\zeta$ is:
	\begin{enumerate}
		\item {\em superattracting} if $R'(\zeta) = 0$;
		\item {\em attracting} if $0 < |R'(\zeta)| < 1$;
		\item {\em repelling} if $|R'(\zeta)|> 1$;
		\item {\em rationally indifferent} if $R'(\zeta)$ is a root of unity; and
		\item {\em irrationally indifferent} if $|R'(\zeta)| = 1$, but $R'(\zeta)$ is not a root of unity.
	\end{enumerate}
\end{defn}
For a fixed point, $\zeta$, we will also call $R'(\zeta)$ the \emph{multiplier} of $\zeta$. We classify fixed points at $\infty$ by conjugacy. Specifically if $\infty$ is a fixed point of $R$ and $g(z)=\frac{1}{z}$, then $gRg^{-1}(0)=gR(\infty)=g(\infty)=0.$ So we can classify $0$ as a fixed point of $gRg^{-1}$ and assign this as the classification of $\infty$ as a fixed point of $R$. 
If $\zeta$ is a point of period $k$, under $R$, we can define the multiplier of the $k$ cycle containing $\zeta$ by viewing $\zeta$ as a fixed point of $R^k$. Notice, this is well defined on a $k$ cycle since $$(R^k(\zeta))'=R'(\zeta)R'(R(\zeta))R'(R^2(\zeta))...R'(R^{k-1}(\zeta)).$$

\subsection{The Julia and Fatou Sets}

\begin{defn}
	The \emph{Fatou set} of $R$ denoted $F$ or $F(R)$ is the maximal open set in $\mathbb{C}_{\infty}$ on which the family of maps $\{R^n\}$ is equicontinuous and the \emph{Julia set}, $J$ or $J(R)$ is the complement of $F$ in $\mathbb{C}_{\infty}$.
\end{defn}
Equivalently, the Fatou set is the maximal open set on which the family of iterates $\{R^n\}$ is normal. When considering the Julia and Fatou sets, we can consider any iterate $R^n$ that we desire without altering $F$ and $J$ since:
\begin{theorem}
	For any non-constant rational map $R$ and any positive integer $n$, $F(R^n)=F(R)$ and $J(R^n)=J(R)$.
\end{theorem} 

\begin{defn}
	A subset $E$ of $\mathbb{C}_{\infty}$ is said to be \emph{forward invariant} if $R(E)=E$, \emph{backward invariant} if $R^{-1}(E)=E$ and \emph{completely invariant} if $E$ is both forward and backward invariant. 
\end{defn}

\begin{theorem}\label{invariant}
	The Fatou set and the Julia set are both completely invariant.
\end{theorem}

\begin{theorem}
	Let $\Omega_0$ and $\Omega_1$ be components of $F(R)$ and suppose that $R$ maps $\Omega_0$ into $\Omega_1$. Then for some integer $m$, $R$ is an $m$-fold map of $\Omega_0$ onto $\Omega_1$.
\end{theorem}
\begin{defn}
	A component $\Omega$ of $F(R)$ is:
	\begin{enumerate}
		\item \emph{periodic} if for some positive integer $n$, $R^n(\Omega)=\Omega$;
		\item \emph{eventually periodic} if for some positive integer $m$, $R^m(\Omega)$ is periodic; and
		
	\end{enumerate}
\end{defn}
\begin{theorem}[Sullivan's No Wandering Domains Theorem]
	\label{Sullivan}
	Every component of the Fatou set of a rational map is eventually periodic. 
\end{theorem}
Therefore if we want to understand the Fatou set of a rational map, we only need to consider components that are forward invariant under some iterate $R^n$ since all other components of the Fatou set must be preimages of these.
\begin{theorem}[Classification of Fatou Components]\label{Classification}
	For a rational  map $R$, a forward invariant component $\Omega$ of the Fatou set, $F(R)$ must be one of the following types:
	\begin{enumerate}
		\item an attracting component if it contains an attracting fixed point of R;
		\item a super-attracting component if it contains a super-attracting fixed point of R;
		\item a parabolic component if there is a rationally indifferent fixed point $\zeta$ on the boundary $\partial \Omega$, and if $R^n \rightarrow \zeta$ on $\Omega$; 
		\item a Siegel Disc if $R:\Omega\rightarrow \Omega$ is analytically conjugate to a Euclidean rotation of the unit disk onto itself;
		\item or a Herman ring if $R:\Omega\rightarrow \Omega$ is analytically conjugate to a Euclidean rotation of some annulus onto itself.
	\end{enumerate}
	
\end{theorem}
	 In this paper we are particularly interested in ($5$). The proof of Thm \ref{Classification} depends on the possible subsequential limit functions on a forward invariant component of $F(R)$.
\begin{defn}
		A function $\psi$ is a limit function on a component $\Omega$ of $F(R)$ if there is some subsequence of $(R^n)$ which converges locally uniformly to $\psi$ on $\Omega$. The class of limit functions on $\Omega$ will be denoted $\mathscr{F}(\Omega)$.
\end{defn}

\begin{lemma}
	If there exists a constant limit function in $\mathscr{F}(\Omega)$ with value $\zeta$, then $\zeta$ is a fixed point of $R^n$ for some $n \geq 1$. 
\end{lemma} 
\begin{theorem}
	Suppose that $\Omega$ is a forward invariant component of $F(R)$ and every function in $\mathscr{F}(\Omega)$ is constant. Then $\mathscr{F}(\Omega)$ contains exactly one function, say with value $\zeta$, where $R(\zeta)=\zeta$ and $R^n\rightarrow \zeta$ locally uniformly on $\Omega$.
\end{theorem}
Moreover, if $\Omega$ is forward invariant and every limit function is constant, $\Omega$ is a (super)attracting or parabolic component.  
\begin{theorem}
	Suppose that $\Omega$ is a forward invariant component of $F(R)$, where deg$(R)\geq 2$, and that $\mathscr{F}(\Omega)$ contains non-constant functions. Then $\Omega$ is either a Siegel disc or a Herman ring. And the two cases can be distinguished by the connectivity of $\Omega$
\end{theorem}
We can also distinguish the two cases by the presence of a fixed point inside $\Omega$.
\subsubsection{Critical Points}
A point $c\in \mathbb{C}_{\infty}$ is a \emph{critical point} for the map $R$ if $R'(c)=0$ or if the multiplier at $\infty$ is $0$. The \emph{valency}, $v_f(z_0)$ or \emph{order} of a function $f$ at a point $z_0$, is the number $k$ such that $\lim\limits_{z\rightarrow z_0}\frac{f(z)-f(z_0)}{(z-z_0)^k}$ exists, is finite and non-zero. Note that $k$ is also the number of solutions of $f(z)=f(z_0)$ at $z_0$.
\begin{theorem}[Riemann-Hurwitz relation]\label{R-H}
	For any non-constant rational map $R$,
	$$\sum_{z\in \mathbb{C}_{\infty}} (v_R(z)-1)=2\deg (R)-2.$$
\end{theorem}
As $v_R(z)=1$ everywhere except critical points of $R$, this tells us that a rational map of degree $d$ has exactly $2d-2$ critical points counting multiplicity. Thus for our maps we will have at most $4$ distinct critical points since we are focusing on the degree $3$ case. Denote the set of critical points of $R$ by $C$ or $C(R)$ and $C^+=\bigcup\limits_{n=0}^{\infty}R^n (C)$ for the forward images. The \emph{postcritical} set,  $\mathit{Cl}({C^+})$,  is the closure of the forward images.
\begin{defn}
	The \emph{immediate basin} of a (super)attracting cycle $\{\zeta_1,...\zeta_q\}$ is the union of components of $F(R)$ containing some $\zeta_j$. 
\end{defn}

\begin{theorem}
	Let $\{\zeta_1,...\zeta_m\}$ be a rationally indifferent cycle for $R$ with multiplier $e^{2\pi i r/q}$ where $r$ and $q$ are coprime. Then there exists an integer $k$, and $mkq$ distinct components $\Omega_1,...,\Omega_{mkq}$ of $F(R)$ such that at each $\zeta_j$, there are exactly $kq$ of these components containing a petal of angle $2\pi/kq$ at $\zeta_j$. Further, $R$ acts as a permutation  $\tau$ on $\{\Omega_1,...\Omega_{mkq}\}$, where $\tau$ is a composition of $k$ disjoint cycles of length $mq$, and a petal based at $\zeta_j$ maps under $R$ to a petal based at $\zeta_{j+1}$ (see e.g., \cite{Beardon} for definitions and details.)
\end{theorem}	

\begin{defn}
	An immediate basin of a rationally indifferent cycle is a cycle of components of $F(R)$ each containing a petal at some point in the cycle. 
\end{defn}

\begin{theorem}\label{attracting}
	Each immediate basin of each attracting, super attracting, or rationally indifferent cycle contains a critical point of $R$. 
\end{theorem}

\begin{theorem}\label{ringcycle}
	Let $\{\Omega_1, ..., \Omega_q\}$ be a cycle of Siegel disks or of Herman rings for a rational map $R$. Then the closure of $C^+(R)$ contains $\bigcup \partial \Omega_j.$
\end{theorem}
Therefore we can discover a lot about $F(R)$ by considering the orbits of the critical points of $R$. And since we have a bound on the number of critical points, this gives us a bound on the total number of forward invariant Fatou components. In fact, a theorem due to Shishikura states a stronger result:
\begin{theorem}[Shishikura]\cite{Surgery} \label{surgery}
	A rational function of degree $d$ has at most $2d-2$ cycles of stable regions, with each Herman ring counted twice and there exist at most $d-2$ Herman rings. 
\end{theorem} 
This theorem implies that maps of degree $3$ can have at most $1$ forward invariant Herman ring.

\section{Dianalytic maps on $\rp$} \label{sec:dianalytic} We give the general form for dianalytic maps as well as a few results about them. 
A $\emph{dianalytic}$ map, $\tilde{f}$ of $\mathbb{RP}^2$ is a map that lifts to an analytic map of the double cover $\mathbb{C}_{\infty}$. We primarily look at $f$ as an analytic map of $\mathbb{C}_{\infty}$ while keeping in mind the additional structure forced on $f$ by being a lift of $\tilde{f}$. Clearly, to project to a dianalytic map, an analytic map must commute with the antipodal map, $\phi(z)=-1/\overline{z}$. Let $\mathfrak{p}:\mathbb{C}_{\infty}\rightarrow\mathbb{RP}^2$ denote the projection map identifying $\mathbb{C}_{\infty}$ as a double cover of $\mathbb{RP}^2$. Then for appropriate $f$,  $\tilde{f}:\mathbb{RP}^2\rightarrow\mathbb{RP}^2$ defined by $\tilde{f}([z])=\mathfrak{p}\circ f\circ \mathfrak{p}^{-1}([z])$ is a well defined dianalytic map on $\mathbb{RP}^2$ where $[z]$ denotes the equivalence class of $z\in \mathbb{C}_{\infty}$ under $\phi$. 

There is a one to one correspondence between dianalytic maps $\tilde{f}:\mathbb{RP}^2\rightarrow \mathbb{RP}^2$ and analytic self maps of  $\mathbb{C}_{\infty}$ (i.e. rational maps) commuting with $\phi$. Additionally in  \cite{BarzaGhisa} it is shown that these rational maps must be of the form:
\begin{equation}\label{general form}
	f(z)=e^{i\theta}\frac{a_0z^{2n+1}+a_1z^{2n}+...+a_{2n+1}}{-\overline{a_{2n+1}}z^{2n+1}+\overline{a_{2n}}z^{2n}-...+\overline{a_0}}, \hspace{.7cm} |a_0|+|a_{2n+1}|\neq 0. 
\end{equation}
\begin{prop}\label{antipodal fixed points}\cite{GHawk} 
	If $f$ is analytic on $\mathbb{C}_{\infty}$ and commutes with $\phi$ and $\zeta$ is a fixed point for $f$ then $\omega=\phi(\zeta)$ is also a fixed point of $f$ with $f'(\omega)=\overline{f'(\zeta)}$ .
\end{prop}
\begin{prop}\label{Fatouu}
	If $f$ is analytic and commutes with $\phi$ (i.e., $f$ is of the form (\ref{general form})) and $F$ is the Fatou set of $f$ then $\Omega\subseteq F$ implies $\phi(\Omega)\subseteq F$.
\end{prop}
\begin{proof}
	Let $\{f_n\}$ be a sequence of iterates of $f$. As $\{f^n\}$ is normal in $\Omega$, $\{f_n\} $ contains a subsequence, $\{f_{n_k}\} $ converging locally uniformly on $\Omega$ to some function $g$. But $f_{n_k}\phi=\phi f_{n_k}$, so,
	\begin{align*}
	\lim\limits_{k\rightarrow\infty} f_{n_k}\phi&=	\lim\limits_{k\rightarrow\infty}\phi f_{n_k}\\
	&=\phi \lim\limits_{k\rightarrow\infty} f_{n_k}\\
	&=\phi(g)
	\end{align*}
	So $f_{n_k}$ converges locally uniformly on $\phi(\Omega)$. Thus $\phi(\Omega)$ is in $F$.
\end{proof}
\begin{theorem}\label{antipodalFatou}
	If $f$ is analytic and commutes with $\phi$ and $\Omega$ is a forward invariant component of $F(f)$ then $W=\phi(\Omega)$ is also a forward invariant component of the same type.
\end{theorem}
\begin{proof}
	By Prop \ref{Fatouu}, $W \subset F$. By definition, if $\Omega$ is a forward invariant component of $F(f)$ then $f(\Omega)=\Omega$. Since $f$ commutes with $\phi$, $f(W)=f\circ\phi(\Omega)=\phi\circ f(\Omega)=\phi(\Omega)=W$. So $W$ is also forward invariant.
	Additionally, $f$ satisfies the hypothesis of Prop \ref{antipodal fixed points}, so if $\zeta$ is a fixed point for $f$ and $\phi(\zeta)=\omega$, then $|f'(\omega)|=|f'(\zeta)|$ and $\omega$ is also a fixed point of $f$ of the same type.
	
	Suppose $\Omega$ is a (super)attracting component. Then $\Omega$ contains a (super)attractive fixed point $\zeta$; this implies $\omega=\phi(\zeta)\in W$ is also a (super)attractive fixed point, so $W$ must be an (super)attracting Fatou component. 
	
	Suppose $\Omega$ is a parabolic component. Then $\Omega$ will have a rationally indifferent fixed point $\zeta$ on $\partial \Omega$ and therefore $\omega=\phi(\zeta)$ will be a rationally indifferent fixed point on $\partial W$. Let $x\in W$. Then $\phi(x)=z\in \Omega$. As $\Omega$ is parabolic, $f^n(z)\rightarrow \zeta$. Since $f$ commutes with $\phi$, $f^n\phi(z)=\phi f^n(z)$ for all $n\in\mathbb{N}$. So $\lim\limits_{n\rightarrow \infty}f^n(x)=\lim\limits_{n\rightarrow \infty} f^n\phi(z)=\lim\limits_{n\rightarrow\infty}\phi f^n(z)\rightarrow \phi (\zeta)=\omega.$ Thus $W$ is a parabolic Fatou component. 
	
	Suppose $\Omega$ is a Siegel disk. If $W$ were a (super)attracting or parabolic component, the arguments above would mean that $\phi(W)=\Omega$ would also have to be (super)attracting or parabolic, thus by Thm \ref{Classification}, $W$ must be a Siegel disk or a Herman ring. As $\Omega$ contains a fixed point $\zeta$, $\omega=\phi(\zeta)$ must also be a fixed point in $W$. So $W$ cannot be a Herman ring and therefore is a Siegel disk.
	
	Suppose $\Omega$ is a Herman ring. If $W$ were a (super)attracting component, parabolic component, or Siegel disk the arguments above would mean that $\phi(W)=\Omega$ would also have to be a (super)attracting component, parabolic component, or a Siegel disk. Since this is not the case, $W$ is also a Herman ring.
\end{proof}

\subsection{Bicritical Rational Maps}
Sue Goodman and Jane Hawkins studied bicritical rational maps  (the map $f$ has exactly $2$ distinct critical points) of the form (\ref{general form}) in \cite{GHawk}. 
This is a fairly strong condition which we relax to obtain new examples here. They showed that every bicritical rational map of odd degree inducing a dianalytic map of $\mathbb{RP}^2$ must be conjugate to a map of the form:
\begin{equation}\label{bicritical}
	f_{\alpha}(z)=\frac{z^n+\alpha}{-\overline{\alpha}z^n+1}, \hspace{.7cm} \arg(\alpha)\in [0,\pi/(n-1)].
\end{equation}
In \cite{GHawk}, they showed that if $f$ is of the form (\ref{bicritical}), $V$ cannot be a Herman ring. This will no longer be the case when we remove the bicritical assumption. Furthermore, if $f_{\alpha}$ has a non-repelling $k-$cycle, $B$, either $B$ is the \emph{only} non-repelling cycle or there are exactly two non-repelling cycles, $B$ and $\phi(B)$. In the first case, $k$ is even and $\mathfrak{p}(B)$ is a $k/2$ cycle in $\mathbb{RP}^2$. In the second case $B\cap\phi(B)=\emptyset$ and $B$ and $\phi(B)$ collapse to a single $k-$cycle in $\mathbb{RP}^2$. These results apply to any function of the form (\ref{bicritical}) with odd degree, however, they had additional results for degree $3$ maps. They generated a parameter space and proved specific dynamical properties for maps with $\alpha$ purely real or purely imaginary. Particularly relevant for us will be the dynamics along the imaginary axis since we will focus on perturbations off this axis. So, the imaginary axis will be a subset of our parameter space. Along the imaginary axis we see a lot of period $2$ behavior (shown in Fig \ref{fig: HawkSpace} in Section $5$). In particular, if $\alpha=ib$ with $b\in\mathbb{R}$ they showed:
\begin{itemize}
	\item For $b<1/\sqrt{2}$, the map $f_{ib}$ has a pair of attracting fixed points. 
	\item For $b=1/\sqrt{2}$, the map $f_{ib}$ the map has $2$ neutral fixed points.
	\item For $b \in (1/\sqrt{2},\sqrt{2})$, the map $f_{ib}$ has two attracting period-$2$ orbits that map to a single period-2 orbit in $\mathbb{RP}^2$.
	\item For $b=\sqrt{2}$, the map $f_{ib}$ has a neutral period $2$ orbit that collapses to a neutral fixed point in $\mathbb{RP}^2$.
	\item for $b>\sqrt{2}$, the map $f_{ib}$ has a attractive period $2$ orbit that collapses to an attracting fixed point.
\end{itemize}

\section{Examples of Herman rings for dianalytic maps}\label{sec:examples}

\subsection{A Perturbation of Degree $3$ Bicritical Dianalytic Maps}\label{section:ourmaps}
Here we focus on a perturbation of the degree $3$ maps in \cite{GHawk} that still satisfy (\ref{general form}) but are no longer bicritical. Namely we consider maps of the form 

\begin{equation}\label{degree3}
	f_{a,b}(z)=\frac{z^3+az^2+b}{-\bar{b}z^3-\bar{a}z+1}, \hspace{.7cm} a,b\in \mathbb{C}\setminus 0.
\end{equation}
This type of map was mentioned briefly in \cite{GHawk}, but not fully explored. By varying the parameters, $a$ and $b$, we find maps with many dynamical properties different from those discussed in \cite{GHawk}. 

We claim that unlike the maps in \cite{GHawk}, some maps of this form have Fatou components that are Herman rings. The two cases that we focus on are maps such that:
\begin{enumerate}
	\item $F$ contains both a Herman ring and attractive periodic components, and
	\item $F$ contains only a Herman ring and its preimages. 
\end{enumerate}
Additionally, since (\ref{degree3}), involves two complex parameters (which results in a $4$ dimensional real parameter space) we focus on $a \in \mathbb{R}$ and $b=\beta i$ for $\beta \in \mathbb{R}$.  Some other subsets of the parameter space and examples can be found in \cite{Rand}.

An example of the first type, the map $\ds f_{1.9,1.5i}(z)=\frac{z^3+1.9z^2+1.5i}{1.5iz^3-1.9z+1}$, can be seen in Fig \ref{fig:1.9and1.5i}. The critical points for this map are $c_1\approx-0.690012 + 1.97138 i$, $c_2\approx1.24652 - 0.4363 i$, $c_3=-1/\overline{c_1}$, and $c_4=-1/\overline{c_2}$. This map has antipodal attractive period two cycles approximated by $ \{-0.0740296 - 1.51629 i, 1.18601 - 0.262465 i\}$ (in blue) and $ \{-0.8038 + 0.177882 i, 0.0321221 + 0.657934 i\}$ (in green). Let $\Delta$ denote the unit disk. Here the green cycle lies inside $\Delta$ and the blue cycle lies outside of $\Delta$. The $2$-cycles are clearly attractive since the modulus of the multiplier of the cycles is easily approximated to be  $0.5277$ (this is the same for both cycles by the antipodal symmetry of $f$). Green points in the image are in the basin of attraction of the green cycle, while light blue points are in the basin of attraction to the blue cycle. The coloration of other points is based on where inside the ring the orbit of the point lands. The gray and black curves show the portion of the postcritical set coming from iteration of $c_1$ and $c_3$ that form the boundary of the Herman ring. Red points do not converge to the ring or periodic cycle
so highlight $J(f)$, and the algorithm is described in \cite{Rand}.

\begin{figure}[h]
	\centering
	\includegraphics[width=10cm]{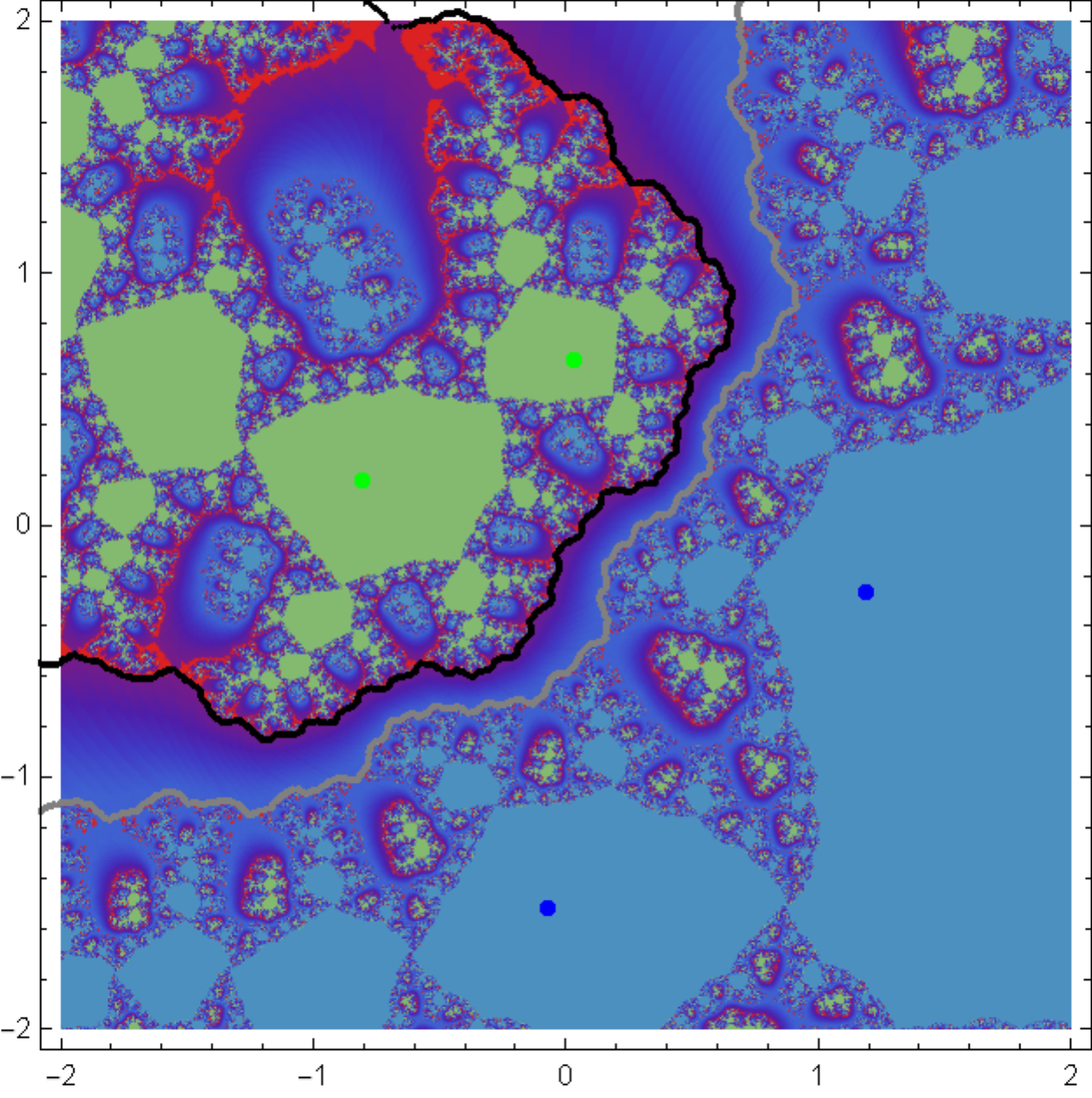}
	\caption{Julia and Fatou set for the map $f_{1.9,1.5i}(z)=\frac{z^3+1.9z^2+1.5i}{1.5iz^3-1.9z+1}$, with attracting period 2 cycles in green and blue.}
	\label{fig:1.9and1.5i}
\end{figure}

The map $ \ds f_{1.8,2i}(z)=\frac{z^3+1.8z^2+2i}{2iz^3-1.8z+1}$ is an example of the second type and can be seen in Fig \ref{fig:1.8and2i}. Here the critical points are $c_1\approx-1.04445 + 1.86507 i$, $c_2\approx1.3711 - 0.767825 i$, $c_3=-1/\overline{c_1}$, and $c_4=-1/\overline{c_2}$. In this image, points are colored according to where in the ring their orbits land, so points that are dark purple, under iteration end up near the upper edge of the ring, lighter points go towards the lower edge of the ring, and so on. The black curves show the postcritical set and each distinct curve corresponds to the orbit a different critical point. Notice that two of the curves lie inside the ring and two along the boundary of the ring. The exterior curves correspond to $c_1$ and $c_3$, while the interior curves correspond to iterates of $c_2$ and $c_4$. Points that are white do not converge within the number of iterations allowed; therefore these are points that are  likely to be in or very close to the Julia set. 
\begin{figure}[h]
	\centering
	\includegraphics[width=10cm]{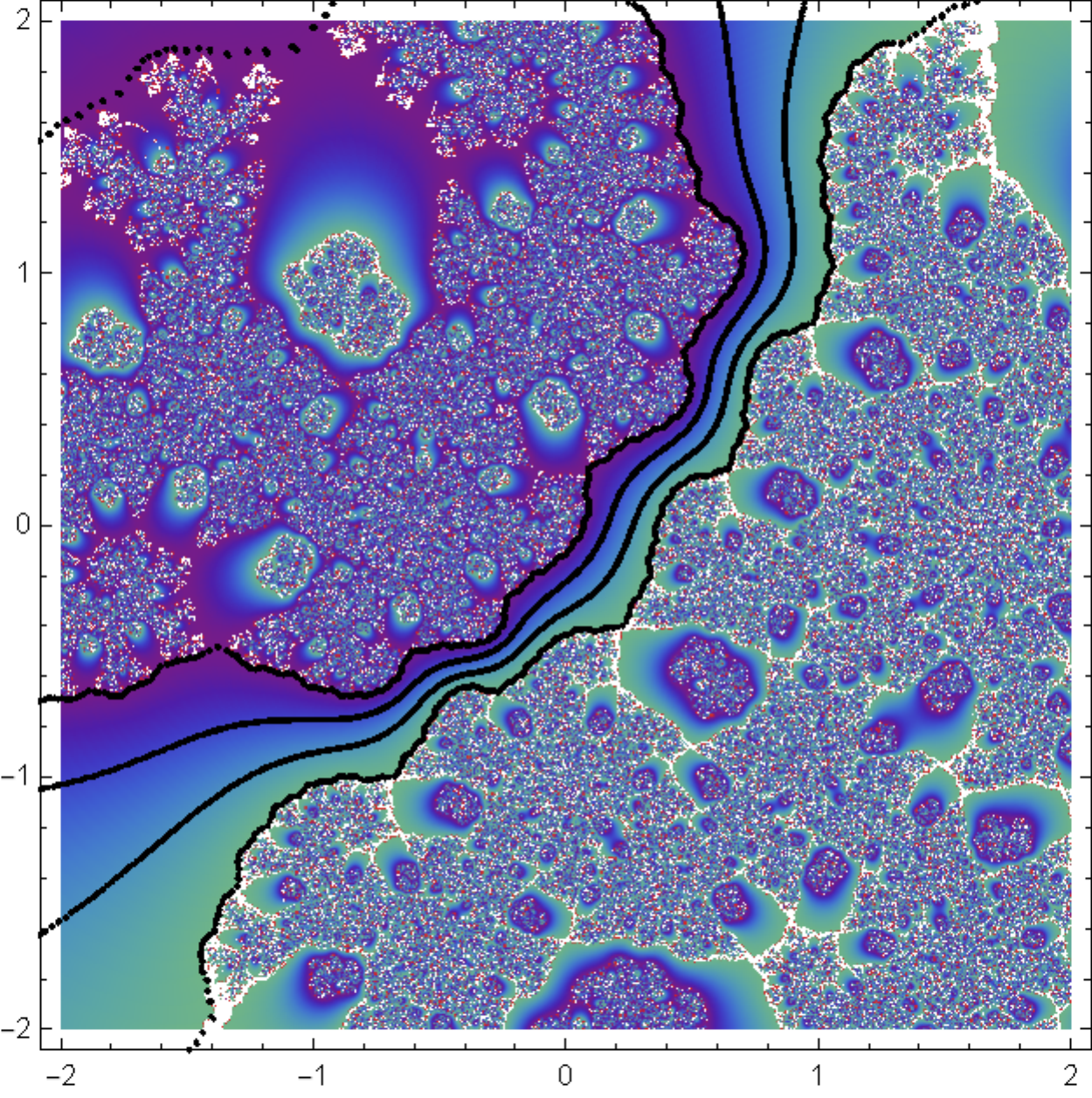}
	\caption{Julia and Fatou set for the map $f_{1.8,2i}(z)=\frac{z^3+1.8z^2+2i}{2iz^3-1.8z+1}.$ }
	\label{fig:1.8and2i}
\end{figure}

These two pictures may be misleading since they appear to contain a band rather than a ring. Due to the antipodal symmetry of our maps anything that occurs inside $\Delta$ (the unit disk) must also be present outside $\Delta$. We make this more rigorous in what follows.  

We note that there are critical points on the boundary of the forward invariant Herman ring shown in Fig \ref{fig:1.8and2i};  we show only the critical orbits and the antipodal critical points  that form the boundary of the Herman ring in Fig \ref{fig2:18and2i}. The local appearance of a right angle is consistent with the map $z \mapsto z^2$ taking perpendicular rays injectively onto a line, since the valency of each critical point is 2 by Thm \ref{R-H}.  There is set of rotation numbers of full measure occurring for a Herman ring that yield a critical point on the boundary. 
  On the other hand, the other two critical points, which end up under iteration inside the ring as shown in Fig \ref{fig:1.8and2i}, cannot be contained in the forward invariant ring.  This is due to the analytic conjugacy to an irrational rotation inside the ring \cite{Herman}, \cite{ Milnor1}.

\begin{figure}[h]
	\centering
	\includegraphics[width=9cm]{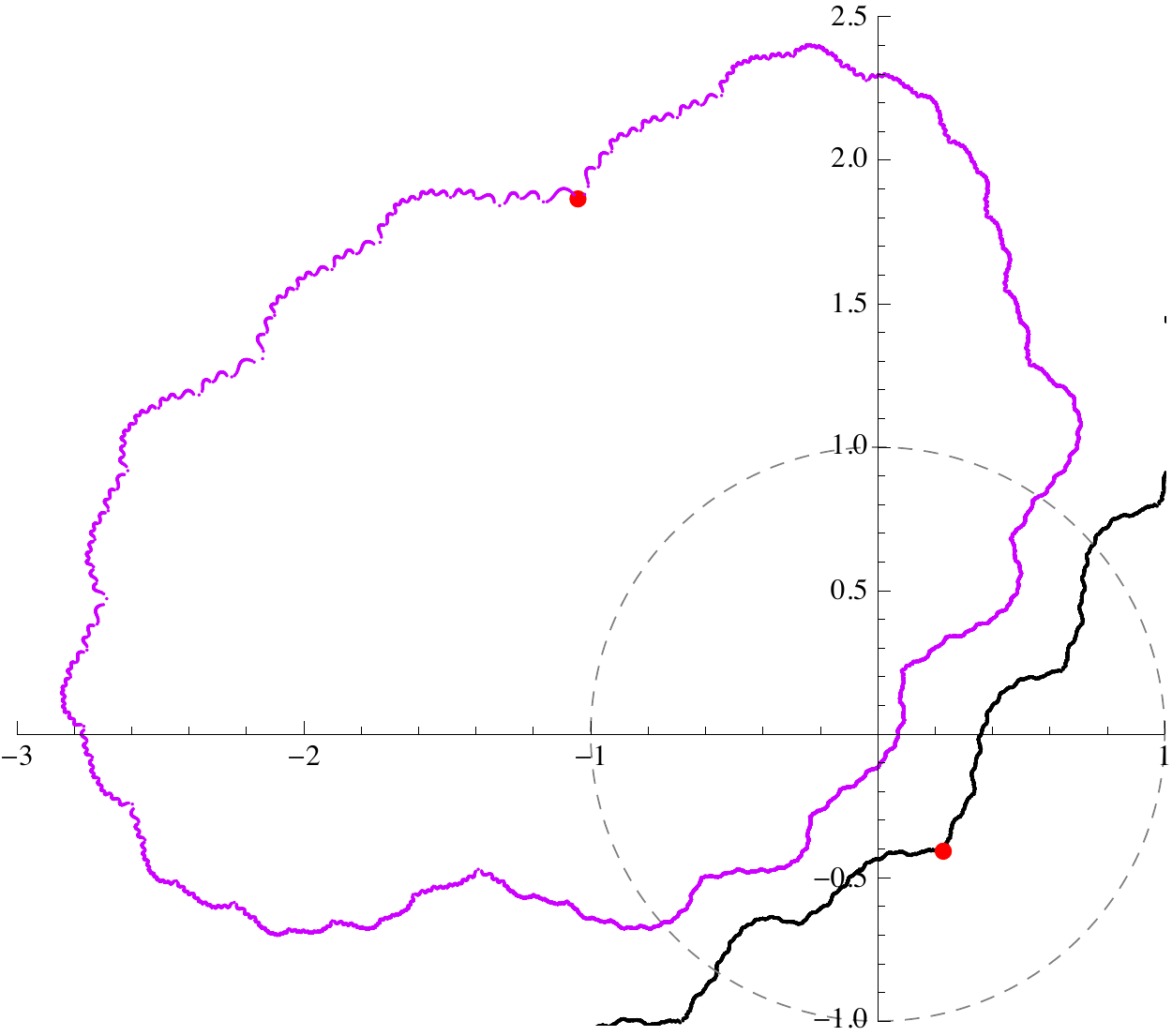}
	\caption{Two antipodal critical orbits and critical points (red) for the map $f_{1.8,2i}(z)=\frac{z^3+1.8z^2+2i}{2iz^3-1.8z+1}.$ }
	\label{fig2:18and2i}
\end{figure}

\newpage
\subsection{Topological properties of $F(f)$}

\begin{lemma}
	Let $f$ be an analytic map such that $f$ commutes with $\phi(z)=-1/\overline{z}$. Suppose there exists a continuous path, $\gamma:[0,1]\rightarrow F(f)\cap \mathit{Cl}(\Delta)$ such that $\gamma(0)=e^{i\theta}$ for some $\theta \in \mathbb{R}$ and $\gamma(1)=-\frac{1}{\overline{\gamma(0)}}=-e^{i\theta}$. Then there exists a loop $\tilde{\gamma}$ containing $\gamma$ such that $\tilde{\gamma}\subset F(f)$ and intersects $\Delta$ only on $\gamma$.
\end{lemma}
\begin{proof}
	Suppose $\gamma_1$ is such a path. Let $\phi(z)=-1/\overline{z}$. Define $\gamma_2=\phi \circ \gamma_1.$ Notice that $\gamma_2:[0,1]\rightarrow \mathbb{C}_{\infty}\setminus \Delta$. Also, $\gamma_2(0)=\phi \circ \gamma_1(0)=\gamma_1(1)$. Consider the path, $$\tilde{\gamma}(t)=\gamma_2 * \gamma_1(t)=\begin{cases}
	\gamma_1(2t)&0\leq t\leq 1/2\\
	\gamma_2(2(t-1/2)) &1/2 < t\leq 1.
	\end{cases}$$
	Then $\tilde{\gamma}$ is a loop in $\mathbb{C}_{\infty}$ since $\gamma_2(1)=\phi \circ \gamma_1(1)=\gamma_1(0).$ Also notice that $\tilde{\gamma}$ lies inside $F(f)$ since for any $z$ such that $\tilde{\gamma}(t)=z$ for some $t$, by construction, either $z$ or $-1/\overline{z}$ must be in the image of $\gamma_1$, which is contained in $F(f).$ As $f$ commutes with the antipodal map, if $-1/\overline{z}\in F(f)$, so is $z.$ Thus $\tilde{\gamma}$ is the desired loop.
\end{proof}
Therefore if $f$ is a rational map commuting with $\phi$ and $\Omega$ is a Fatou component for $f$ containing two antipodal points and a path between them in $\mathit{Cl}(\Delta)$, then $\Omega$ must also contain a path in $\mathbb{C}_\infty\setminus\Delta$ joining the antipodal points (which will be on $\partial\Delta$). This tells us that the bands of Fatou components shown in Fig \ref{fig:1.9and1.5i} and Fig \ref{fig:1.8and2i} must continue outside of the unit disk and are actually annuli. Figure \ref{fullring} shows the same map from Fig \ref{fig:1.9and1.5i} zoomed out so that the full ring can be seen. The coloration is the same in both figures. Due to computational error, while it looks like the gray curve is living in a sea of blue, it actually forms the boundary of the Herman ring. We present the method for generating these pictures in more detail in \cite{Rand}.

\begin{figure}[h]
	\centering
	\includegraphics[width=10cm]{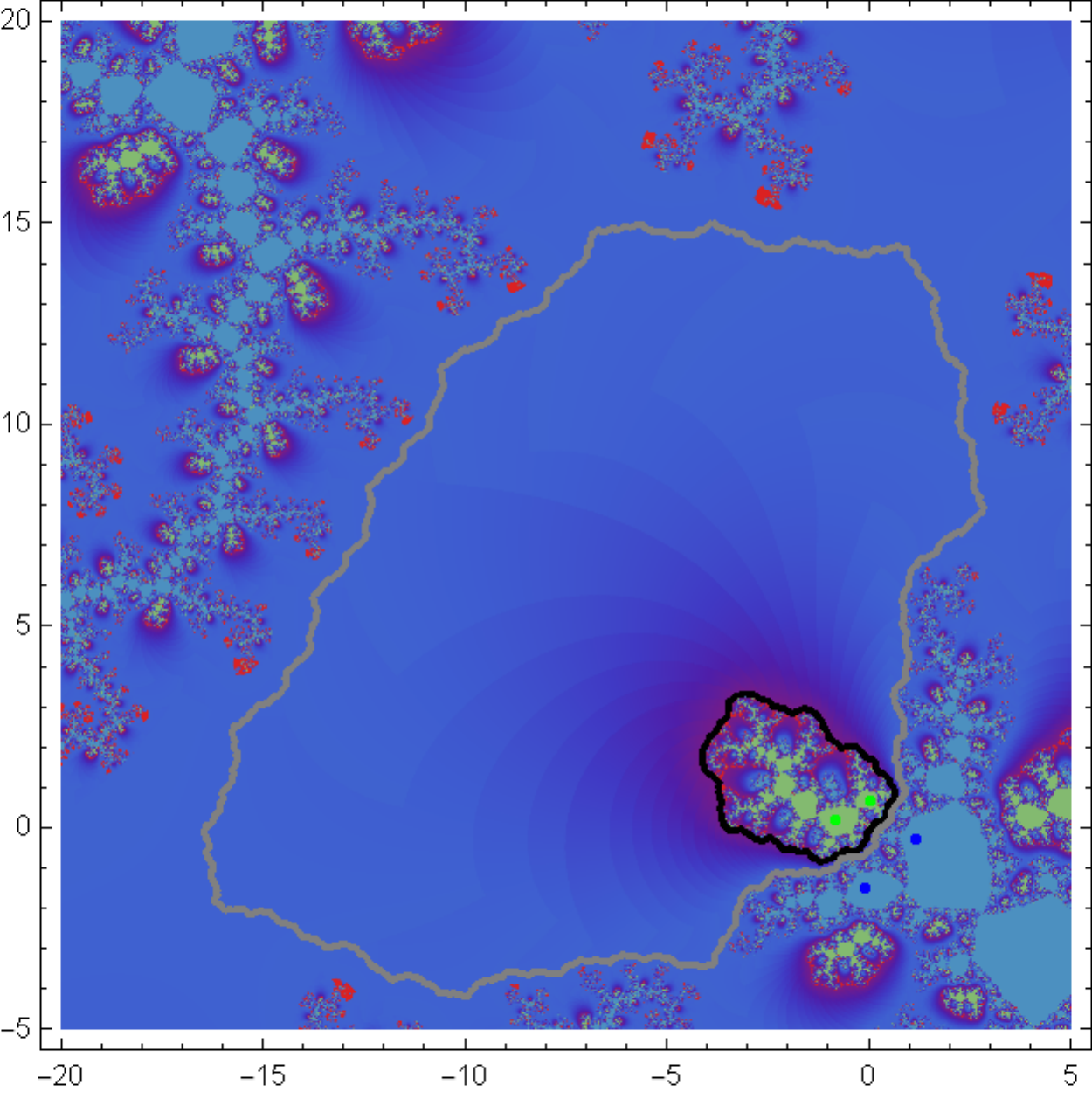}
	\caption{Julia and Fatou set for the map $f_{1.9,1.5i}(z)$ zoomed out to see the full ring.}
	\label{fullring}
\end{figure}

In the arguments above, we assume that a component $\Omega$ of the Fatou set, $F(f)$, contains antipodal points. The next lemma provides a partial converse. 
\begin{lemma}\label{H}
	For a map $f$ of the form (\ref{degree3}), if $F(f)$ has a forward invariant Herman ring $H$, then $\phi(H)=H$.
\end{lemma}
\begin{proof}
	Since $H$ is a Herman ring, $\phi(H)$ must also be a Herman ring by Thm \ref{antipodalFatou}. But by Thm \ref{surgery}, since $\deg(f)=3$, $f$ can have at most one forward invariant Herman ring. Thus $H=\phi(H)$.
\end{proof}

It is also the case that if $F(f)$ contains a Herman ring, then it is forward invariant.  The next result is from \cite{preprint}.
\begin{theorem} \label{prop:bbm}  If $f$ is any degree $3$ rational map that commutes with $\phi$ and contains a Herman ring $H$, then $f(H)=H$.

\end{theorem}

\begin{remark}  A detailed proof of Thm \ref{prop:bbm} appears in \cite{preprint}.  There is an additional claim appearing in (\cite{preprint}, Thm 7.1)  that says that $H$ must separate $0$ and $\infty$ under the hypotheses of Thm \ref{prop:bbm};  however that statement is not proved.  Moreover we have an example of the form (\ref{degree3}), which is not the form studied in \cite{preprint}, where it seems to be false since $0, \infty \in H$ appears to hold.  This is illustrated in Fig \ref{fig:0inf}.
\end{remark}

\begin{figure}[h]
	\centering
	\includegraphics[width=10cm]{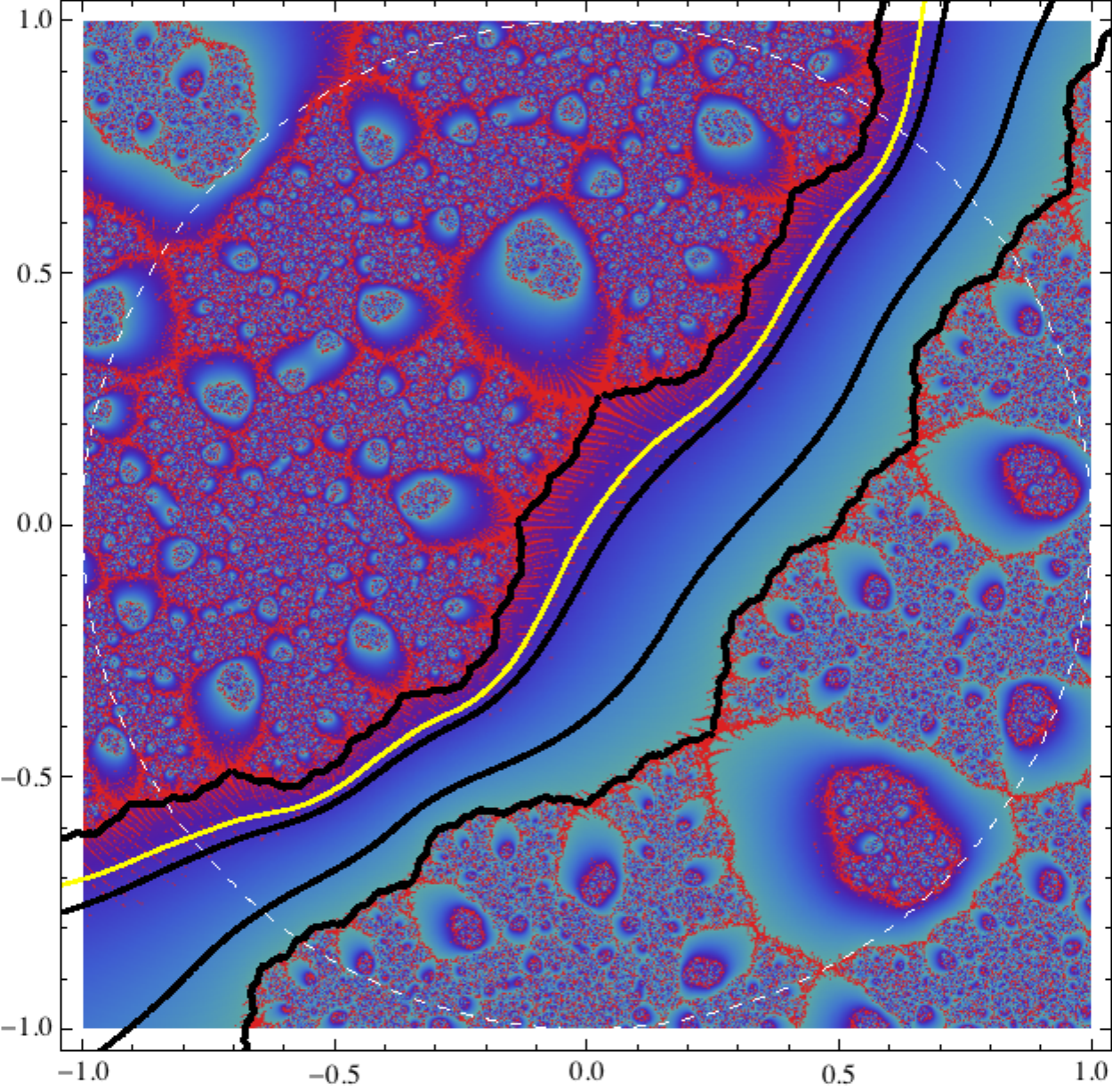}
	\caption{Julia and Fatou set for the map $f_{2,2.3i}(z)$ with orbits of the critical points in black and the orbit of $0$ between 2 critical orbits, in yellow.}
	\label{fig:0inf}
\end{figure}

We continue to explore the possible dynamics for our maps.

\begin{theorem}
	Suppose $f$ is an analytic map of the form (\ref{degree3}) and $F(f)$ has a (forward invariant) Herman ring $H$. If $F(f)$ also has a periodic component $U$ associated with a (super)attractive or rationally indifferent $k-$cycle $B=\{\zeta,f(\zeta),...,f^{k-1}(\zeta)\}$, then exactly one of the following holds:
	\begin{enumerate}
		\item $B=\phi(B)$; in this case $k$ must be even and $B$ is the only (super)attracting or rationally indifferent cycle for $f$.
		\item $B\cap \phi(B)=\emptyset$; in this case, there are exactly two (super)attracting or rationally indifferent cycles for $f$ and they are antipodal to each other on $\mathbb{C}_{\infty}$.
	\end{enumerate}
\end{theorem}

\begin{proof}
	By Thm \ref{ringcycle},  the closure  $\mathit{Cl}(C^+(f))$ must contain $\partial H$. Thus some critical point $c_1$, of $f$ must have infinite forward orbit contained in $\partial H$. It immediately follows that the critical point, $c_2=\phi(c_1)$ must also have infinite forward orbit contained in $\partial H$ since $\phi\circ f^m=f^m \circ \phi$ for all $m \in \mathbb{N}$. Since $\deg(f)=3$, by Thm \ref{R-H}, $f$ has at most $4$ distinct critical points, so there are at most two remaining critical points.  \\
	If $U$ is a periodic component of $F(f)$ associated with a (super)attracting or rationally indifferent $k$-cycle, $B$, there is a critical point, $c_3$ in the immediate attracting basin (where $c_3\neq c_1$ and $c_3\neq c_2$). If there is another (super)attracting or rationally indifferent cycle for $f$, the only remaining critical point, $c_4=\phi(c_3)$ must be attracted to it.
	By Prop \ref{antipodal fixed points}, the second cycle must be $\phi(B)$ and $B\cap \phi(B)=\emptyset$.
	
	If there is a pair of points $\zeta,\phi(\zeta)\in B$, then $\phi(\zeta)=f^m(\zeta)$, for some $1\leq m\leq k-1$. Since $\phi\circ f^m=f^m\circ\phi$ for all $m\in\mathbb{N}$, $B=\phi(B)$. In this case, both remaining critical points $c_3$ and $c_4=\phi(c_3)$ must be attracted to the same cycle. As we have accounted for all of the critical points, no other (super)attracting or rationally indifferent cycles can occur. For each $\zeta\in B$, $B$ must also contain $\phi(\zeta)$. Since $\phi$ has no fixed points and $\phi(\phi(\zeta))=\zeta$, points in $B$ come in pairs. Thus $B$ has an even period.
\end{proof}

In the first case, $B$ will collapse to a $k/2$-cycle in $\mathbb{RP}^2$ and in the second case, $B$ and $\phi(B)$ will collapse to a single $k$-cycle in $\mathbb{RP}^2$.

In the previous sections, we primarily view maps of the form (\ref{degree3}) as analytic functions on $\mathbb{C}_{\infty}$; however their form is chosen so that $\tilde{f}$ can be viewed as a dianalytic map on $\mathbb{RP}^2$. So, it is natural to wonder about the structure of the Fatou components when projected onto $\mathbb{RP}^2$. In our examples, if there is a Herman ring $H$ it is forward invariant by Thm \ref{prop:bbm}; thus in the arguments that follow, we assume that $f$ is a map of the form (\ref{degree3}) and there is exactly one forward invariant Herman ring in $F(f)$. By Lemma \ref{H}, $\phi(H)=H$. Either $H$ cuts through the unit disk ($\Delta$) and contains pairs of antipodal points but $\partial\Delta \nsubseteq H$ or $\partial\Delta \subset H$. In either case when this component of the Fatou set is projected, there is a M\"obius strip contained in $\mathfrak{p}(H)$ on $\mathbb{RP}^2$. In $\mathbb{C}_{\infty}$, notice that $\mathbb{C}_{\infty}\setminus H$ consists of two disjoint simply connected  components and the center curve of $H$  projects to a generator of the fundamental group on $\mathbb{RP}^2$, $\pi_1(\mathbb{RP}^2)\cong \mathbb{Z}/2\mathbb{Z}$. Any preimage of $H$ (besides $H$ itself) cannot intersect $H$, therefore any preimage $\tilde{H}\neq H$ lies entirely in one of the components and the center curve of $\tilde{H}$ is homotopic to the identity. So, each preimage $\tilde{H}$ maps to an annulus $\mathfrak{p}(\tilde{H})$ in $\mathbb{RP}^2$ and the forward invariant component, $\mathfrak{p}(H)$, is non-orientable. We have outlined a proof of the following:
\begin{prop}
	
Let $f$ be an analytic map of the form (\ref{degree3}). If $F(f)$ contains a Herman ring, $H$, then $\mathfrak{p}(H)$ contains a M\"obius strip. Additionally for each preimage $\tilde{H}$ of $H$ such that $\tilde{H}\neq H$, $\mathfrak{p}(\tilde{H})$ is an annulus in $\mathbb{RP}^2$.
\end{prop}

\subsection{How to Identify a Herman Ring for Dianalytic Maps of Degree 3} 

Suppose we have a Fatou component $\Omega$; we would like to determine what type it is. We know that $\Omega$ is periodic or eventually periodic. For simplicity, assume that $\Omega$ itself is periodic with period $n$, so $\Omega$ will be forward invariant under iteration of $R^n$. Then, by Thm \ref{Classification} $\Omega$ must be
an attracting component,
a super-attracting component,
a parabolic component,
a Siegel Disk,
or a Herman Ring. 
In the first three cases $\mathit{Cl}(\Omega)$ contains a fixed point of $R^n$ and thus we see a nonrepelling cycle of period $n$. In the last two cases $\Omega$ is part of a periodic cycle of Siegel disks or Herman rings. The following results are useful for determining what types of Fatou components we have for maps of the form (\ref{degree3}).

\begin{theorem}\label{Julia}\cite{Beardon}
	For any rational map $R$ of $\deg\geq 2$, either $J(R)=\mathbb{C}_{\infty}$ or $J$ has empty interior. 
\end{theorem}
Let $(X,\mathfrak{F},\mu)$ be a probability space (i.e., a measure space with $\mu(X)=1$).
\begin{defn}\cite{JHawk}
	 Let $f$ be a nonsingular map, (i.e., $\forall S \in \mathfrak{F}$, $f^{-1}(S)\in \mathfrak{F}$ and $\mu( f^{-1}(S))=0$ if and only if $\mu(S)=0$), $f$ is \emph{ergodic} if for all $S\in \mathfrak{F}$ such that $S$ is backward invariant, $\mu(S)=0$ or $\mu(S)=1$.   
\end{defn}
Let $m_2$ denote the normalized surface area measure on $\mathbb{C}_{\infty}$.

\begin{theorem}\label{Thm:Ergodic}\cite{McMullen}
	For any rational map $R$ on $\mathbb{C}_{\infty}$, either
	\begin{enumerate}
		\item $J=\mathbb{C}_{\infty}$ and the action of $R$ on $\mathbb{C}_{\infty}$ is $m_2$-ergodic, or 
		\item the spherical distance $d(R^n(x),\mathit{Cl}(C^+))\rightarrow 0$ for $m_2$ almost every $x$ in $J$ as $n \rightarrow \infty$ .
	\end{enumerate}
\end{theorem}
\subsection{Typical Examples} We discuss two examples, the first map has an attractive period $2$ cycle and a Herman ring and the second is a map where the only Fatou components are a Herman ring and its preimages. The two examples in this subsection are the same as those from Section \ref{section:ourmaps} and are typical maps of the form (\ref{degree3}) possessing Herman rings. In previous sections, we claim without proof that these maps have Fatou components that are Herman rings. In this section, for each type of example we justify the belief that the region appearing to be a Herman ring is actually a Herman ring, giving partial arguments relying on experimental results. Since these examples are typical of the maps we study, these arguments translate to other examples as well. 

\textbf{Example 1:}  Consider the map: $\ds f_{1.9,1.5i}(z)=\frac{z^3+1.9z^2+1.5i}{1.5iz^3-1.9z+1}.$ The Julia set ($J$) and Fatou set ($F(f_{1.9,1.5i})$) for this map can be seen in Figures \ref{fig:1.9and1.5i}, \ref{fullring}, and \ref{fig:random}.
As with many of our maps, it is clear that we have antipodal attractive orbits. This can be seen by calculating the fixed points of $R^n$ and finding the points such that the multiplier is less than $1$, or simply by following the orbit of critical points (by Thm \ref{attracting}). The map $f_{1.9,1.5i}$ has two antipodal cycles of period two, each attracting a critical point. The closures of the orbits of the other two critical points appear to be forming Jordan curves. While it is theoretically possible that these are actually periodic orbits with extremely high period we assume that they are not.

 Let $U$ denote the ring bounded by the two Jordan curves. Due to the attracting $2$-cycle, $J$ is not the whole sphere. Thus by Thm \ref{Julia}, $J$ has empty interior. Since $U$ contains interior points, $U \not\subset J$ so $F(f)\cap U\neq \emptyset$. In Fig \ref{fig:random}, we show the orbits of two points randomly generated from inside $U$ and the unit square; these orbits are overlaid on the image from Fig \ref{fig:1.9and1.5i}. 
 (In particular the cyan curves are the orbits of $x_1\approx 0.477204 + 0.366227i$ and $x_2\approx 0.820551 + 0.991023 i$.)
  From the robustness of these orbits under successive runs, we can deduce that iterates of randomly chosen points in $U$ stay in $U$. 
 
 \begin{figure}[h]
 	\centering
 	\includegraphics[width=10cm]{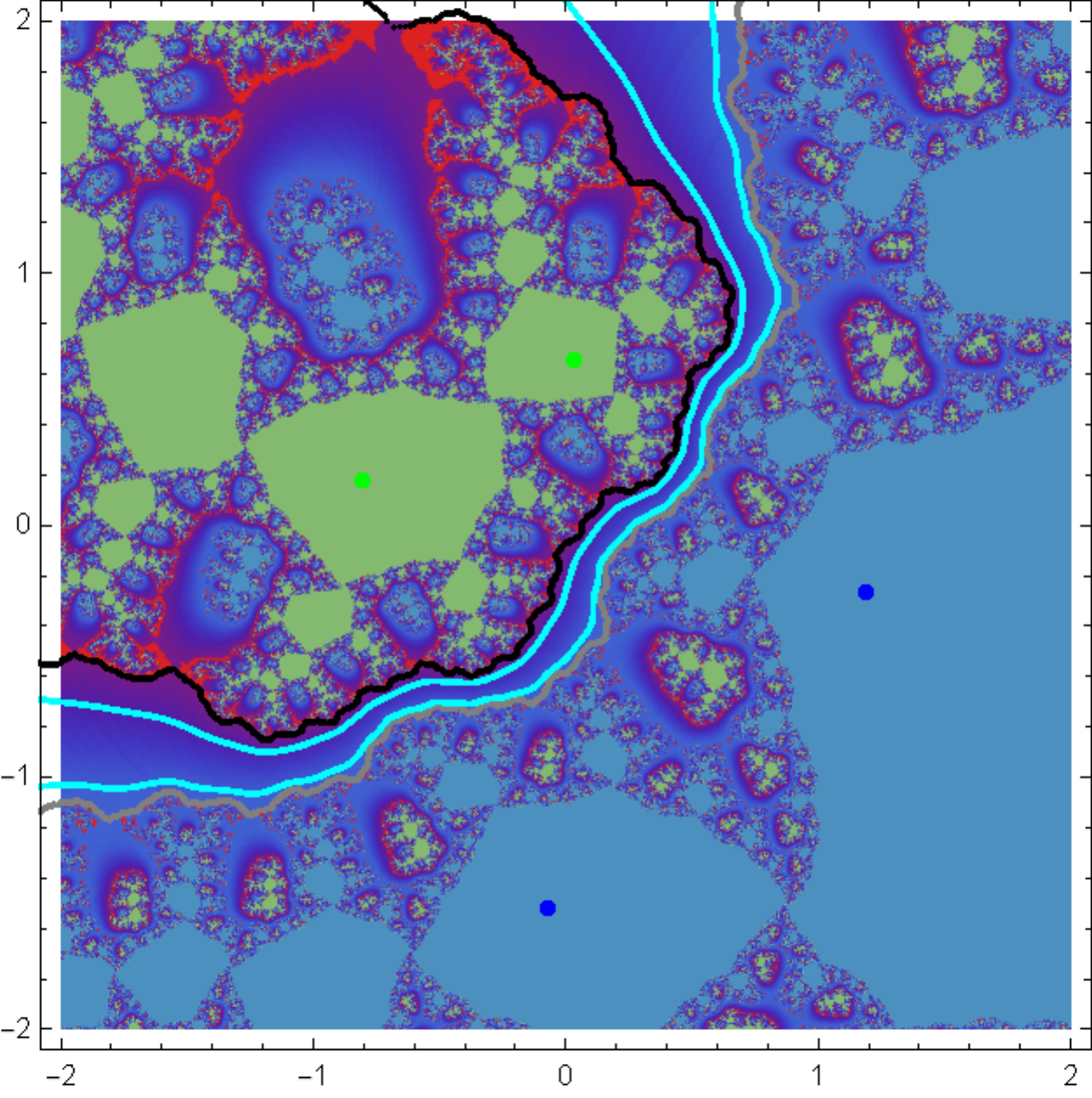}
 	\caption{Julia and Fatou set for the map $f_{1.9,1.5i}(z)$ with orbits of two random points inside the ring shown. }
 	\label{fig:random}
 \end{figure}
We now argue that randomly chosen points such as $x_1$ and $x_2$ are not all in the Julia set. By the presence of an attracting cycle, we established that $J\neq \mathbb{C}_{\infty}$. Therefore, $f$ must satisfy condition $(2)$ in Thm \ref{Thm:Ergodic}. So, if $m_2(J)= \alpha >0$ then $m_2$ almost every $x$ chosen randomly from $J$ satisfies $d(f^n(x),\mathit{Cl}(C^+))\rightarrow 0$. However, as shown in Fig \ref{fig:random}, $d(R^n(x),\mathit{Cl}(C^+))$ stays bounded away from $0$ for each randomly chosen $x\in U$. Therefore for almost every $x\in U$, $x\not \in J$ which implies that for almost every $x\in U$, $x\in F(f)$.

Since $m_2$ almost every $x$ in $U$ is in the Fatou set, it is left to determine what type of Fatou components are contained in $U$. For each randomly chosen $x$ in $U$, we can see that $x$ is not attracted to one of the period $2$ cycles, so the component of the Fatou set containing $x$ is not a preimage of one of the attracting period $2$ components. Additionally, $x$ is not attracted to a different (super)attracting or parabolic cycle since this would require a critical point to also be attracted to the same cycle. Thus, $U$ cannot contain a (super)attracting or parabolic cycle. 

We are left with the case that $U$ contains a Herman Ring or Siegel disk. Thm \ref{ringcycle} implies that the boundary of a Siegel disk or Herman ring must be contained in $\mathit{Cl}(C^+(R))$, for this map, this would mean the boundary of the disk or ring must lie on the black and/or gray curves in Fig \ref{fig:1.9and1.5i}, Fig \ref{fullring}, and Fig \ref{fig:random}. So all of $U$ must be contained in a single ring or disk. Also, note that Siegel disks and Herman rings cannot contain critical points for this would contradict that $f_{1.9,1.5i}$ is a homeomorphism on the disk or ring.  If we had a Siegel disk, since our maps are antipode preserving, there must be two disjoint disks, one bounded by the black curve and the other by the gray curve. This leaves us to conclude that the interior of the Siegel disk would have to contain the periodic orbit (and also the critical point attracted to it). Therefore we cannot have a Siegel disk and $U$ must be a Herman ring.

\textbf{Example 2:} Consider the map: $\ds f_{1.8,2i}(z)=\frac{z^3+1.8z^2+2i}{2iz^3-1.8z+1}.$ The Julia and Fatou sets for $f_{1.8,2i}$ can be seen in Fig \ref{fig:1.8and2i} 
(and Fig \ref{fig:curves}). We can argue similarly that (assuming we do not have an attractive or parabolic orbit of extremely high period) we must have a Herman ring. While we focus on $f_{1.8,2i}$ we can translate these arguments to any map of the form (\ref{degree3}) that produces similar orbits.

If $F(f_{1.8,2i})$ does not contain attractive or parabolic components, then it either contains a Siegel disk, Herman ring, or is empty. We note that if it is a cycle of Siegel disks, the cycle must be of period $\leq 2$ by looking at the postcritical set (by Thm \ref{ringcycle}). The lack of a neutral fixed point or period two orbit rules out this case; it remains to rule out the case $J=\mathbb{C}_{\infty}$. 

Let $V$ be the region bounded by $\mathit{Cl}\Big(\bigcup\limits_{n=1}^{\infty}f^n(c_1)\Big)$ and $\mathit{Cl}\Big(\bigcup\limits_{n=1}^{\infty}f^n(c_4)\Big)$ for $c_1$ and $c_4$ as defined for this map in Section \ref{section:ourmaps}; so $V$ is the region bounded by the black curves in Fig \ref{fig:curves}. We note that $c_4$ does not lie on the boundary of the annulus in the image; thus $V$ is only part of the annulus bounded by the postcritical set. 
\begin{figure}[h]
	\centering
	\includegraphics[width=10cm]{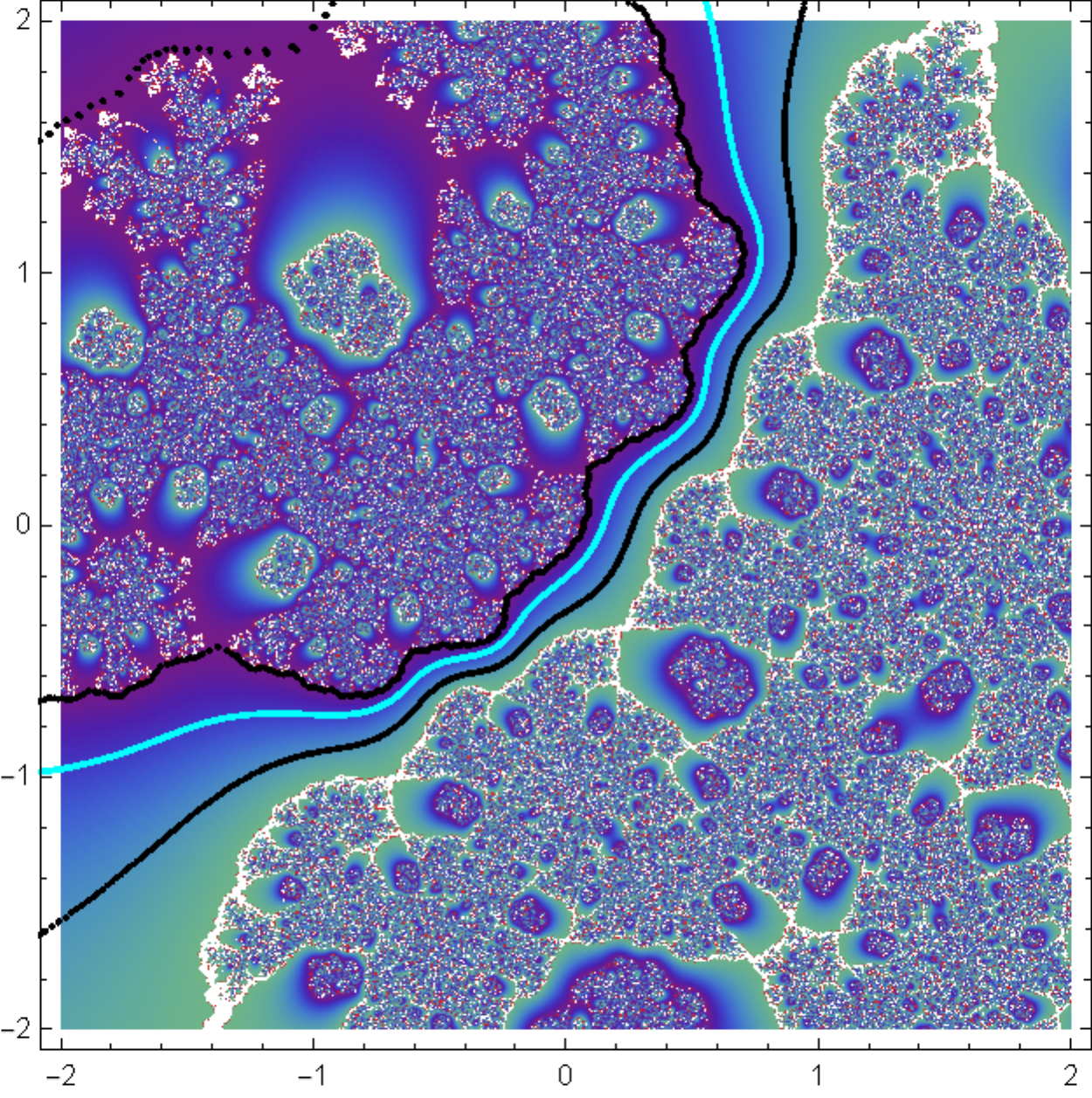}
	\caption{Julia and Fatou set for the map $f_{1.8,2i}(z)$ with orbits of two critical points and an orbit of a random point between the critical orbits.}
	\label{fig:curves}
\end{figure}
In Fig \ref{fig:curves}, we also show the orbit of a randomly chosen point $x$ inside $V$. In particular, the orbit of the point $x\approx 0.152142 + 0.00815055 i$ is shown in cyan. By iterating random points such as $x$, we see that $V$ appears to be forward invariant since iterates of points in $V$ seem to stay in $V$. We therefore assume that $V$ is forward invariant. With this assumption,  $V\subset\tilde{V}=\bigcup\limits_{n=1}^{\infty}f^{-n}(V)$. Then $f^{-1}(\tilde{V})=\tilde{V}$, so $\tilde{V}$ is an invariant subset of $\mathbb{C}_{\infty}$ and so is $U=\mathbb{C}_{\infty}\setminus \tilde V$. Clearly $\tilde{V}$ and $U$ are both sets of positive measure (since they both contain open sets). Thus we have two disjoint completely invariant sets of positive measure and $f_{1.8,2i}$ is not $m_2$-ergodic. 

We have eliminated the first case in Thm \ref{Thm:Ergodic}, so $f_{1.8,2i}$ must satisfy the second condition of the theorem.  As we want to show that $F(f)$ is non-empty, we assume $m_2(J)>0$. Thus if a random point $x$ is chosen from $J$, the set of positive measure, then $d(R^n(z),\mathit{Cl}(C^+))\rightarrow 0$. However we can see by plotting orbits of arbitrarily chosen random $x$, that the iterates stay bounded away from $\mathit{Cl}(C^+)$. Therefore if $m_2(J)>0$, this contradicts Thm \ref{Thm:Ergodic}, so $F(f)\neq \emptyset$ and the forward invariant component $F_0$ must be a Herman ring.

\section{Parameter spaces}\label{sec:parameters}
We discuss the parameter space for maps of the form $\ds f_{a,bi}(z)=\frac{z^3+az^2+bi}{biz^3-az+1}$ with $a,b\in \mathbb{R}$. Each point $(a,b)$ in the parameter space corresponds to the map $f_{a,bi}(z)$. We prove that the parameter space is symmetric about the $x$ and $y$ axes and compare two different methods for generating the parameter space. We also compare this parameter space to the one studied in \cite{GHawk}. The algorithms used in this section are discussed more in the master's project of the second author. 

\begin{figure}[h]
	\centering
	\includegraphics[width=10cm]{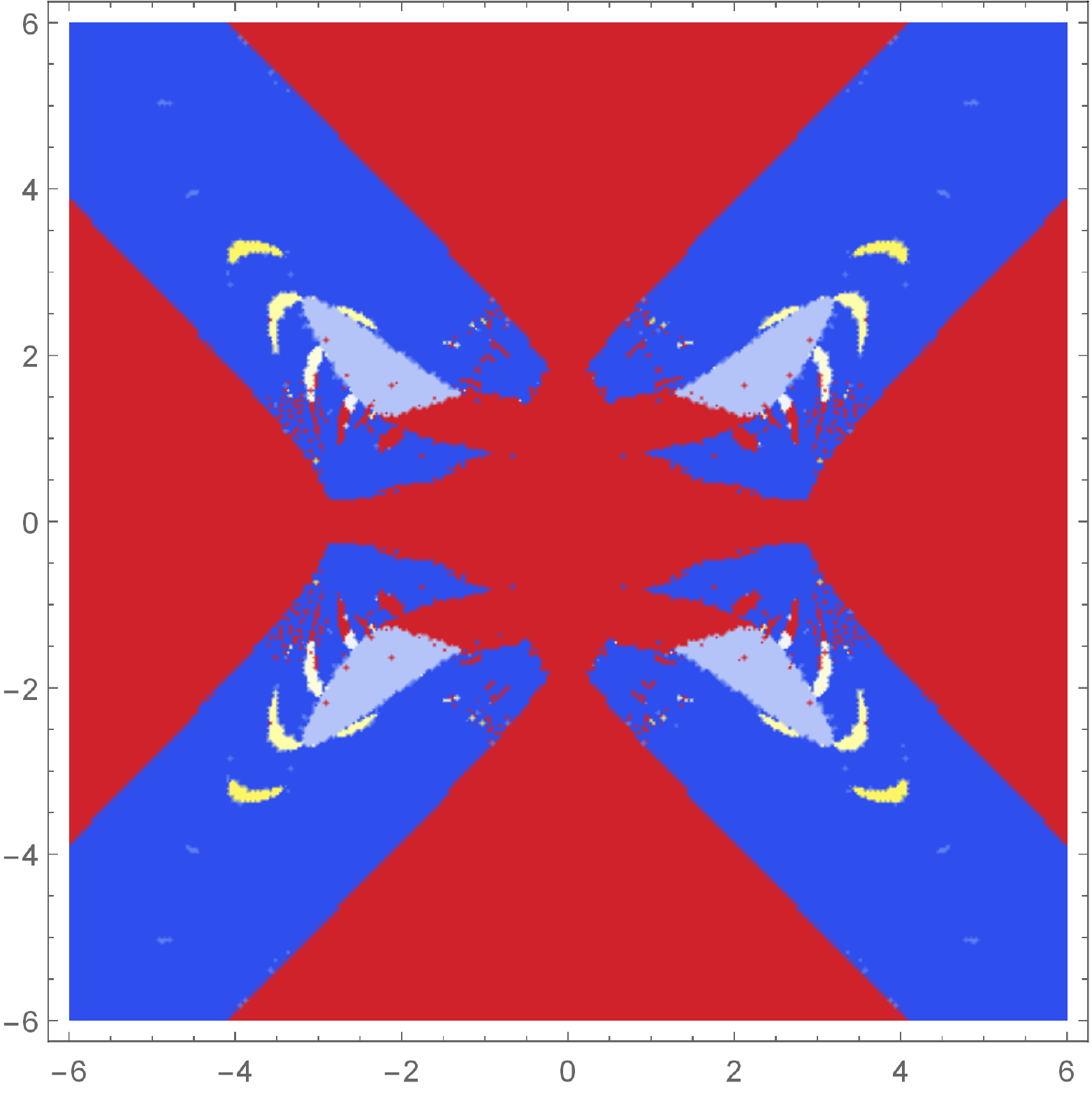}
	\caption{Parameter space for  $f_{a,bi}(z)=\frac{z^3+az^2+bi}{biz^3-az+1}$ generated by a derivative algorithm.}
	\label{fig:pspace}
\end{figure}

The parameter space for these maps can be seen in Fig \ref{fig:pspace}. We observe that for maps of this form, we get symmetry across both axes. To be more precise:

\begin{theorem}
	If $f_{c,di}(z)$ is of the form (\ref{degree3}) with $c,d \in \mathbb{R}$, then $f_{c,di}(z)$ is conjugate via a conformal map (either orientation preserving or reversing) to a map $f_{a,bi}(z)$ of the same form with $a\geq 0$ and $b\geq 0$.
\end{theorem}
\begin{proof}

	Let $\phi_1(z)=\overline{z}$, $\phi_2(z)=-z$, and $\phi_3(z)=-\overline{z}$.   For each $c,d \in \C$ not both $0$, conjugating $f_{c,di}(z)$ by one  of these maps yields the result.
	
\end{proof}

\begin{figure}[h]
	\centering
	\includegraphics[width=10cm]{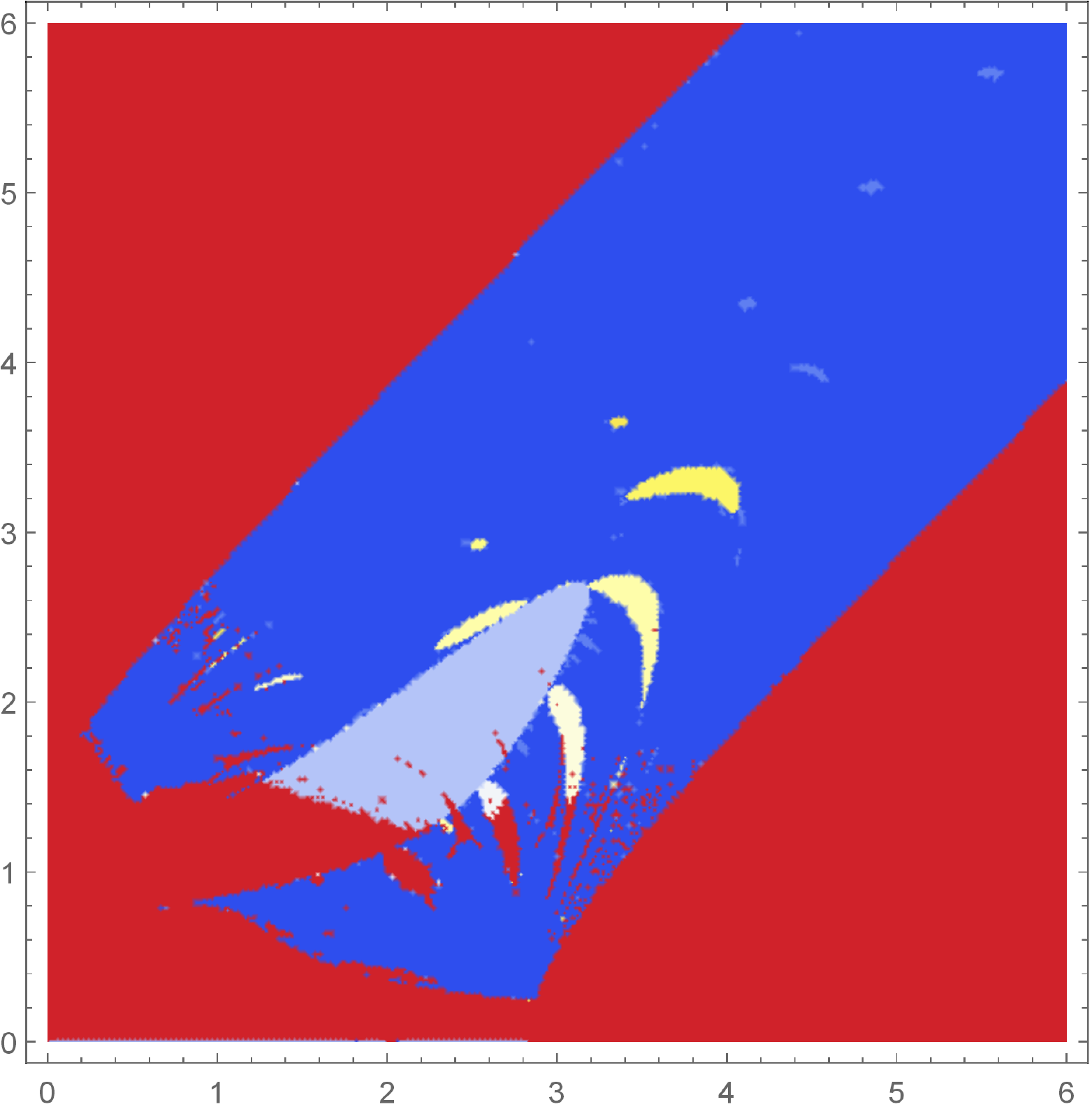}
	\caption{Parameter space for  $f_{a,bi}(z)$ zoomed into the first quadrant.}
	\label{fig:pspacezoomed}
\end{figure}

To visualize the parameter space, we again utilize dynamical properties of critical orbits. A zoomed in version of the first quadrant of Fig \ref{fig:pspace} can be seen in Fig \ref{fig:pspacezoomed}. Since we are interested in maps with Herman rings, we color any point where all critical points tend toward attractive cycles red. Points in dark blue indicate that no critical point for that map tends toward an attractive cycle; this is where we expect to find maps with Fatou set consisting of only Herman rings and their preimages. All other colors represent maps for which two critical points go to attracting cycles and two do not. The coloration of these points corresponds to the period of the attractive cycle(s), in particular, the lightest blue regions correspond to maps with period 2 cycles and a Herman ring. As the color changes from light blue to white to yellow, this corresponds to an increase in the period of the attractive orbit, with the brightest yellow corresponding to maps with an attractive period 12 orbit. The other specks of medium blue that are lighter than the dark blue but darker than the period 2 region, correspond to maps where two critical points converge to something periodic with period greater than 13. For Fig \ref{fig:pspace} and Fig \ref{fig:pspacezoomed} an algorithm involving derivative products  attributed to Buff and Henriksen is used. For more information on the algorithm, we refer to the master's project \cite{Rand}.

In this parameter space, there is still some question of what is happing in the red region in Fig \ref{fig:pspacezoomed}. As this is a perturbation off the imaginary axis in the parameter space in \cite{GHawk}, one might expect some similarities in their parameter space. We include the parameter space from \cite{GHawk} in Fig \ref{fig: HawkSpace} for comparison. In the image, black points correspond to maps with attractive fixed points, gray points correspond to maps with attractive period $2$ orbits, and other colors indicate higher periods. So, we expect that near the imaginary axis in our parameter space, we will also see period one and two behavior.

\begin{figure}[h]
	\centering
	\includegraphics[width=10cm]{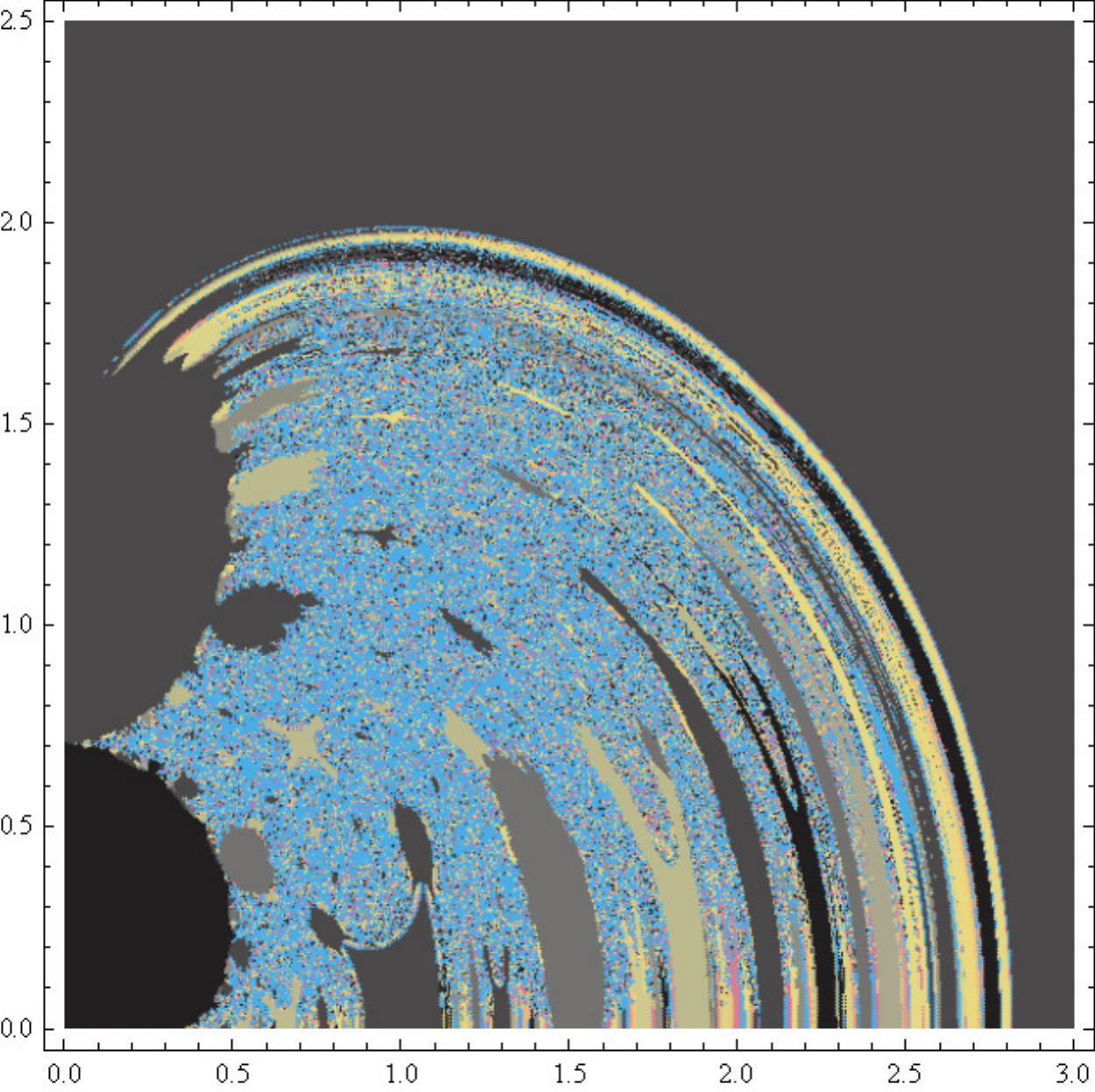}
	\caption{Parameter space for $f_{\alpha}(z)=\frac{z^3+\alpha}{-\overline{\alpha}z^3+1}$ from \cite{GHawk}. }
	\label{fig: HawkSpace}
\end{figure}

To this end, we also generate a parameter space to determine the period of attractive cycles for parameters that correspond to maps without Herman rings. In preliminary attempts, it appears that when all critical points converge to a periodic cycle, all cycles are of the same period. We therefore simplify the image and the algorithm by coloring points according to the lowest period found. 

In Fig \ref{fig: Attractive PSpace}, dark blue regions indicate that no critical point converge to an attractive cycle, medium blue indicates that at least one critical point is attracted to a fixed point, and light blue indicates that at least one critical point converges to a period $2$ orbit. Continuing into the yellows, reds, and whites these are orbits of higher period. 
\begin{figure}[h]
	\centering
	\includegraphics[width=10cm]{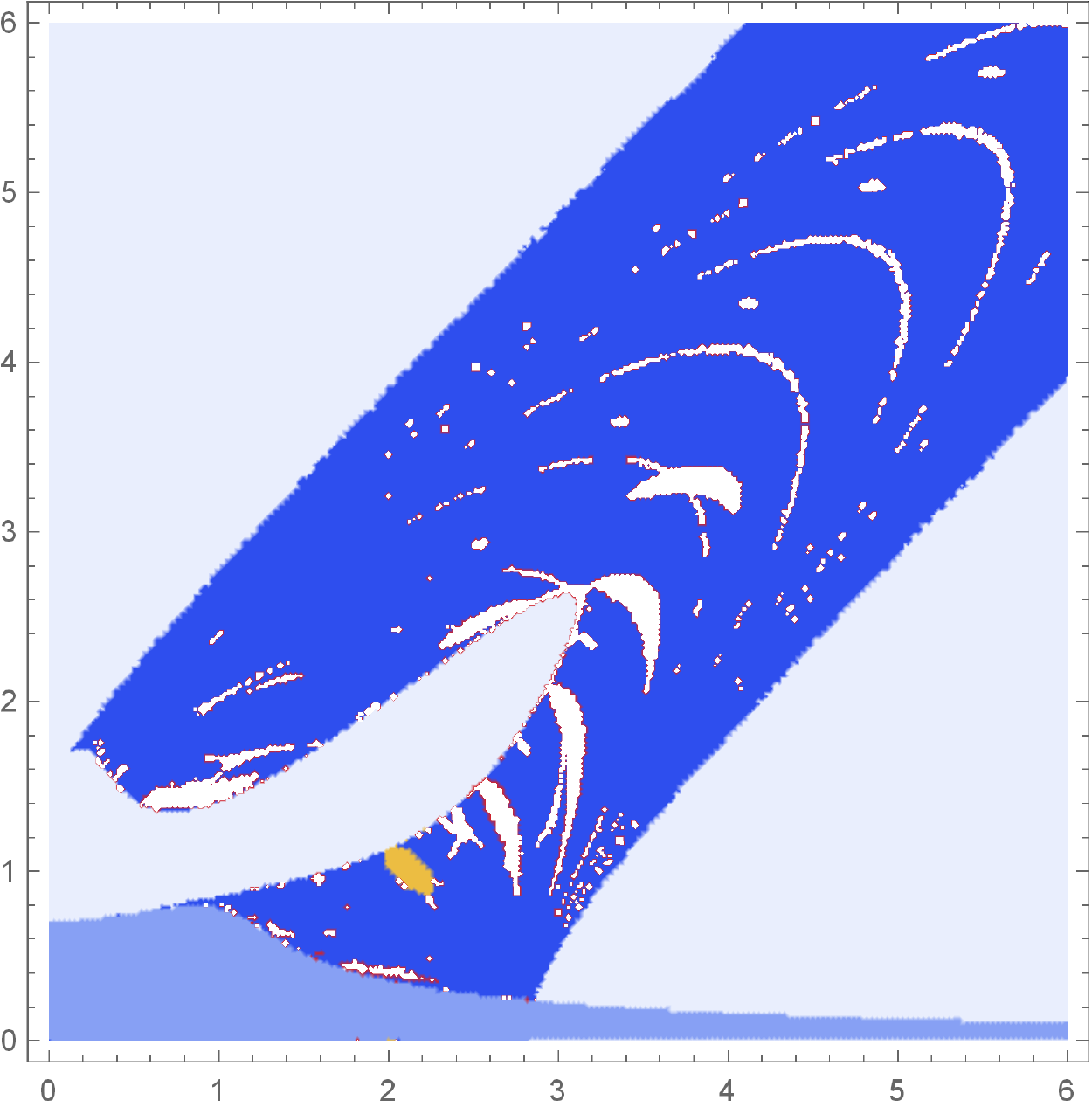}
	\caption{Alternate Parameter Space for $f_{a,bi}(z)$ generated by a convergence algorithm.}
	\label{fig: Attractive PSpace}
\end{figure}

As we would expect, parameters near the imaginary axis correspond to maps with attractive fixed points when $b$ is small and maps with attractive period two points when $b$ is larger. This bifurcation appears to occur around $b=1/\sqrt{2}$, which is consistent with the results form Goodman and Hawkins and Fig \ref{fig: HawkSpace}.

By comparing the two parameter space methods, (Fig \ref{fig:pspacezoomed} and Fig \ref{fig: Attractive PSpace}), we can see that in general for a parameter pair, $(a,b)$, the two methods find the same types of behavior. In particular, if a point is red in Fig \ref{fig:pspacezoomed}, the same point in Fig \ref{fig: Attractive PSpace} will be some color that corresponds to a map for which a critical point is attracted to a periodic cycle. Furthermore, if $(a,b)$ is a point in Fig \ref{fig:pspacezoomed} for which the derivative algorithm finds that only one of the antipodal pairs of critical points converges to a periodic cycle with period $k$, the convergence algorithm finds the same value for the period of the cycle. However there are some places where the second method, found ``attractive" orbits where the first method did not. To see this, look at the dark blue regions; for both algorithms, dark blue points indicate that no critical points converge to a periodic cycle. The dark blue region in Fig \ref{fig:pspacezoomed} is more uniform than the dark blue region in Fig \ref{fig: Attractive PSpace}: by this we mean that the dark blue strip in Fig \ref{fig: Attractive PSpace} has more regions of other colors than the same strip in Fig \ref{fig:pspacezoomed}. This means that the convergence algorithm finds attractive behavior for maps where the derivative algorithm does not. For example, for the map $f_{4.167, 4.06i}$ the convergence algorithm, determines that the map has a period $13$ cycle, while the derivative algorithm finds the map to have no attractive behavior. This discrepancy is due to the algorithms and is discussed more in \cite{Rand}.

There are other interesting maps in related parameter spaces that which we mention briefly.  The map $\ds f_{1 + .4 i,1.5}(z)=\frac{z^3+(1 + .4 i)z^2+1.5}{-1.5z^3-(1 - .4 i)z+1}$, whose
 Julia set is shown in Fig \ref{Period3},  has a Herman ring and two period $3$ orbits. 
\begin{figure}[h!]
	\centering
	\includegraphics[width=10cm]{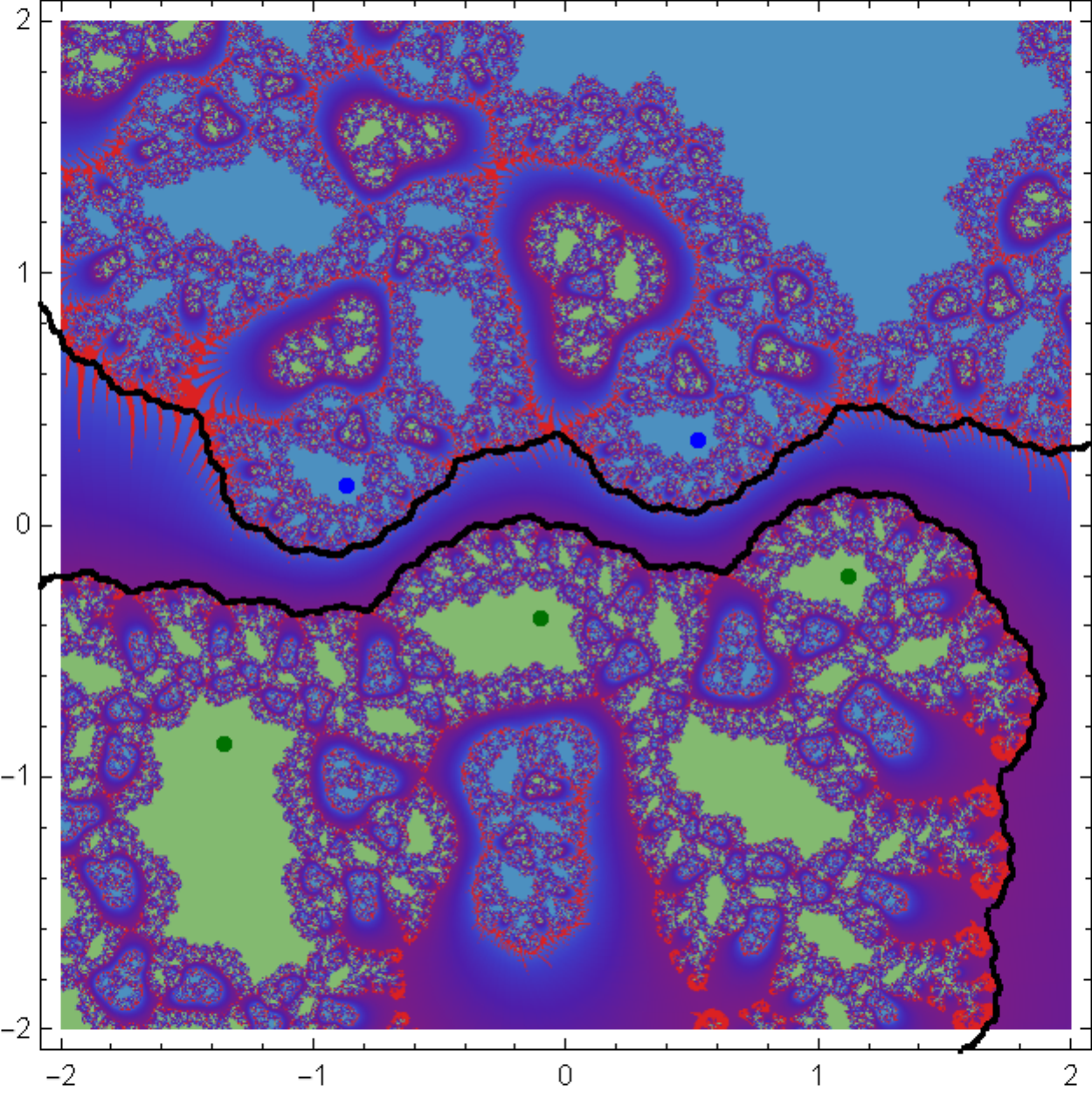}
	\caption{Julia and Fatou set for the map $f_{1 + .4 i,1.5}(z)=\frac{z^3+(1 + .4 i)z^2+1.5}{-1.5z^3-(1 - .4 i)z+1}$  shown with black critical orbits and attracting 3-cycles in blue and green.}
	\label{Period3}
\end{figure}
The coloration is similar to previous images in which blue points are attracted to the dark blue orbit, green points are attracted to the dark green orbit and points that map to the ring are colored purples to blues. While it looks like there are only two points in the blue periodic cycle, the third point in that orbit lies outside of the plot range. This map is a bit different from the others in that the periodic orbits do not lie entirely on one side of the unit circle. That is, two of the blue periodic points and one of the green periodic points are inside $\Delta$.

We have examples with two period $6$ cycles separated by a Herman ring, as can be seen in Fig \ref{Period6} for the map $\ds f_{2.606,1.507i}(z)=\frac{z^3+2.606z^2+1.507i}{1.507iz^3-2.606z+1}$. The image was generated similarly to the other Julia sets, however there is more red in the image because the convergence happens more slowly.
\begin{figure}[h!]
	\centering
	\includegraphics[width=10cm]{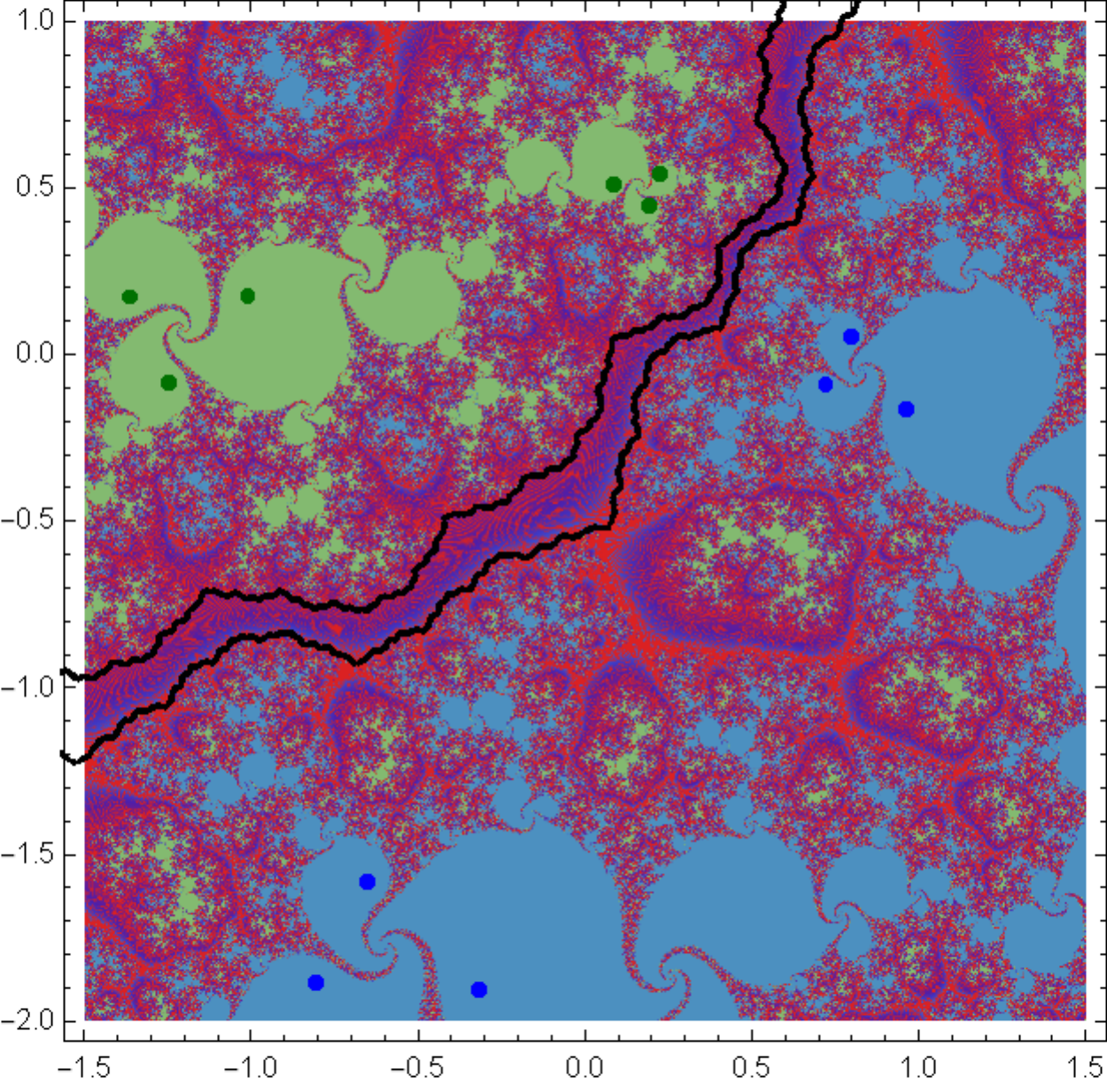}
	\caption{Julia and Fatou set for the map $f_{2.606,1.507i}(z)=\frac{z^3+2.606z^2+1.507i}{1.507iz^3-2.606z+1}$ shown with black critical orbits and attracting 6-cycles in blue and green.}
	\label{Period6}
\end{figure}

\nocite{*}
\bibliography{Hawkins_Randolph}
\bibliographystyle{amsplain}

\end{document}